\documentclass[11pt,a4paper]{amsart}
\usepackage{amsmath,amssymb,amsthm, amsfonts,enumerate,mathtools,pinlabel,float,stmaryrd}
\usepackage[mathscr]{eucal}
\usepackage{amsbsy}
\usepackage{bm}
\usepackage{tikz}
\usetikzlibrary{arrows}
\usetikzlibrary{decorations.markings}
\usepackage{graphicx}
\usepackage{caption}
\usepackage{subcaption}
\newtheorem{theorem*}{Theorem}
\newtheorem{corollary*}{Corollary}

\newtheorem{theorem}{Theorem}[section]
\newtheorem{lemma}[theorem]{Lemma}
\newtheorem{proposition}[theorem]{Proposition}

\theoremstyle{definition}

\newtheorem{definition}[theorem]{Definition}
\newtheorem{example}[theorem]{Example}

\newtheorem{remark}[theorem]{Remark}

\numberwithin{equation}{section}
\theoremstyle{plain}

\newtheorem{corollary}[theorem]{Corollary}

\numberwithin{equation}{section}

\newcommand{\Mod}{\mathrm{Mod}}

\newcommand{\Teich}{\mathrm{Teich}}
\newcommand{\inj}{\mathrm{inj}}
\newcommand{\sys}{\mathrm{sys}}
\newcommand{\Fix}{\mathrm{Fix}}

\newcommand{\C}{\mathcal{C}}
\newcommand{\p}{\mathcal{P}}
\newcommand{\N}{\mathcal{N}}

\newcommand{\U}{\mathcal{U}}

\renewcommand{\O}{\mathcal{O}}
\DeclareMathOperator{\lcm}{lcm}

\renewcommand{\l}{\operatorname{{\llbracket}}}
\renewcommand{\r}{\operatorname{{\rrbracket}}}
\newcommand{\lp}{\operatorname{{\llparenthesis}}}
\newcommand{\rp}{\operatorname{{\rrparenthesis}}}

\makeatletter
\@namedef{subjclassname@2020}{\textup{2020} Mathematics Subject Classification}
\makeatother
\begin{document}

\title[Geometric realizations of cyclic actions on surfaces - II]{Geometric realizations of \\ cyclic actions on surfaces - II}

\author{Atreyee Bhattacharya}
\address{Department of Mathematics\\
Indian Institute of Science Education and Research Bhopal\\
Bhopal Bypass Road, Bhauri \\
Bhopal 462 066, Madhya Pradesh\\
India}
\email{atreyee@iiserb.ac.in}
\urladdr{https://sites.google.com/iiserb.ac.in/homepage-atreyee-bhattacharya/home?authuser=1}

\author{Shiv Parsad}
\address{School of Mathematics and Computer Science\\
Indian Institute of Technology Goa\\
Goa College of Engineering Campus\\
Farmagudi, Ponda-403401, Goa\\
India}
\email{shiv@iitgoa.ac.in}

\author{Kashyap Rajeevsarathy}
\address{Department of Mathematics\\
Indian Institute of Science Education and Research Bhopal\\
Bhopal Bypass Road, Bhauri \\
Bhopal 462 066, Madhya Pradesh\\
India}
\email{kashyap@iiserb.ac.in}
\urladdr{https://home.iiserb.ac.in/$_{\widetilde{\phantom{n}}}$kashyap/}



\maketitle

\begin{abstract}
 Let $ \Mod(S_g)$ denote the mapping class group of the closed orientable surface $S_g$ of genus $g\geq 2$. Given a finite subgroup $H$ of $\Mod(S_g)$, let $\Fix(H)$ denote the set of fixed points induced by the action of $H$ on the Teichm\"{u}ller space $\Teich(S_g)$.  When $H$ is cyclic with $|H| \geq 3$, we show that $\Fix(H)$ admits a decomposition as a product of two-dimensional strips at least one of which is of bounded width.  For an arbitrary $H$ with at least one generator of order $\geq 3$, we derive a computable optimal upper bound for the restriction $\sys : \Fix(H) \to \mathbb{R}^+$ of the systole function. Furthermore, we show that in such a case, $\Fix(H)$ is not symplectomorphic to the Euclidean space of the same dimension. Finally, we apply our theory to recover three well known results, namely: (a) Harvey's result giving the dimension of $\Fix(H)$, (b) Gilman's result that $H$ is irreducible if and only if the corresponding orbifold is a sphere with three cone points, and (c) the Nielsen realization theorem for cyclic groups.
\end{abstract}

\section{Introduction}
\label{sec:intro}
Let $S_g$ be a closed orientable surface of genus $g\geq 2$ and, let $\text{Mod}(S_g)$ denote the mapping class group of $S_g$.  Given a finite subgroup $H \leq \Mod(S_g)$, let $\Fix(H)$ denote the set of fixed points induced by the natural action of $H$ on the Teichm\"{u}ller space $\Teich(S_g)$. The Nielsen realization problem asks whether $\Fix(H) \neq \emptyset$, for an arbitrary finite subgroup $H \leq \Mod(S_g)$. While Kerckhoff~\cite{SK1} answered this in the affirmative,  Harvey~\cite[Theorem 2]{H1} showed that $\Fix(H) \approx \hat{i} (\Teich(S_g/H))$, where $\Teich(S_g/H)$ is defined in the sense of Bers~\cite{LB}, and $\hat{i}$ is the natural embedding induced by the branched cover $S_g \to S_g/H$ (identifying $H$ with a group of self-homeomorphisms of $S_g$). 

Recently, the second and third authors~\cite{PRS} proved a structure theorem for realizing a finite order mapping class as an isometry of a hyperbolic structure on $S_g$. This package had two parts: 
\begin{enumerate}[(1)]
\item realizing an irreducible periodic mapping class (whose Nielsen representative has a fixed point) called a \textit{spherical Type 1 mapping class}, and
\item a complete description of how all reducible \textit{non-rotational} (i.e. that are not realized as surface rotations) periodic mapping classes are built (or realized) inductively from spherical Type 1 mapping classes.
\end{enumerate}
\noindent  In particular, it was shown that a spherical Type 1 mapping class can be realized as the rotation of a polygon with a suitable side-pairing, as illustrated in Figure~\ref{fig:c14-action} below. 
\begin{figure}[htbp]
\labellist
\small
\pinlabel $g$ at 120 36
\pinlabel $b$ at 280 36
\pinlabel $a$ at 200 21
\pinlabel $c$ at 345 93
\pinlabel $f$ at 58 93
\pinlabel $e$ at 29 160
\pinlabel $d$ at 375 162
\pinlabel $d$ at 26 237
\pinlabel $e$ at 377 237
\pinlabel $a$ at 200 378
\pinlabel $g$ at 280 365
\pinlabel $b$ at 123 365
\pinlabel $f$ at 345 310
\pinlabel $c$ at 59 310
\pinlabel $\frac{2\pi}{14}$ at 235 205
\endlabellist
\centering
\includegraphics[width = 25 ex]{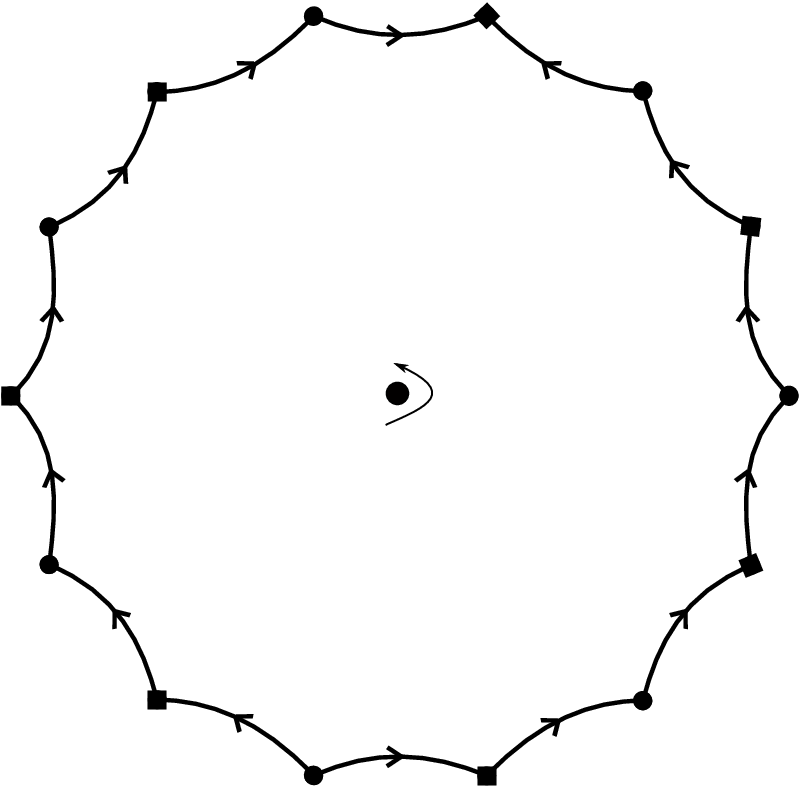}
\caption{An irreducible order $14$ element in $\Mod(S_3)$.}
\label{fig:c14-action}
\end{figure}
The inductive process of building a reducible non-rotational periodic mapping class involved two key constructions: 
\begin{itemize}
\item $\alpha$-\textit{compatibility} - Pasting together a pair of spherical Type 1 actions across boundary components induced by deleting invariant cyclically permuted disks around points in a pair of compatible orbits of size $\alpha$, which are orbits where the induced local rotation angles are equal, as illustrated in Figure~\ref{fig:2-compatibility} below. 

\begin{figure}[htbp]
\centering
\begin{subfigure}{.5\textwidth}
\labellist
\small
\pinlabel $\frac{2\pi}{6}$ at 120 90
\endlabellist
  \centering
  \includegraphics[width=.5\linewidth]{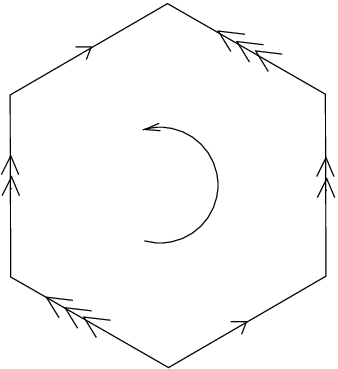}
  \caption{\tiny An irreducible order $6$ element $h  \in \Mod(S_1)$.}
  \label{fig:sub1}
\end{subfigure}%
\begin{subfigure}{.5\textwidth}
\labellist
\small
\pinlabel $\frac{2\pi}{6}$ at 51 22
\pinlabel $\frac{10\pi}{6}$ at 97 22
\endlabellist
  \centering
  \includegraphics[width=.8\linewidth]{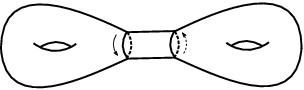}
  \caption{\tiny A $1$-compatibility between $h$ and $h^5$.}
  \label{fig:sub2}
\end{subfigure}
\caption{A $1$-compatibility realizing a order $6$ element in $\Mod(S_2)$. }
\label{fig:2-compatibility}
\end{figure}

\noindent This notion also includes the compatibility across orbits of size $\alpha$ within the same surface, called a \textit{self $\alpha$-compatibility}.
\item \textit{Addition (resp. deletion) of a $g_1$-permutation component} - For $g_1 \geq 1$, pasting (resp. removal) of a cyclical permutation of $n$ copies of $S_{g_1,1}$ to (resp. from) the action induced by $h$ on $S_{g,n}$ by removing invariant cyclically permuted disks around an orbit of size $n$, as shown in Figure~\ref{fig:g1-perm} below. 
\begin{figure}[htbp]
\labellist
\small
\pinlabel $\frac{2\pi}{6}$ at 130 95
\endlabellist
\centering
\includegraphics[width = 20 ex]{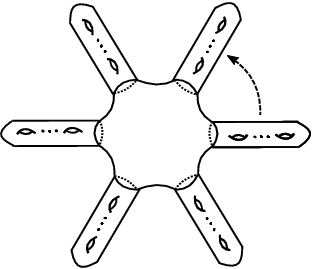}
\caption{\small Addition of a $g_1$-permutation component to a sphere rotation by $2\pi/6$.}
\label{fig:g1-perm}
\end{figure}
\end{itemize}

Given any cyclic subgroup $H \leq \Mod(S_g)$ of order $n$, the Nielsen-Kerckhoff theorem says that one may also regard $H$ as a cyclic subgroup of $\text{Homeo}^+(S_g)$ generated by an orientation-preserving map $h$ of order $n$. Throughout this paper, we will refer to both the mapping class represented by $h$, and the group it generates, interchangeably, as a $C_n$-action on $S_g$. Furthermore, $h$ has a \textit{corresponding orbifold} given by $\O_h \approx S_{g}/\langle h \rangle$. In Section~\ref{sec:cyclic_irreducibles}, we extend the results in~\cite{PRS} to surface rotations of order $n \geq 3$ (see Theorem~\ref{thm:arb_cyc_real}). We also prove that an arbitrary $C_n$-action ($n \geq 3$) $h$ admits a canonical decomposition that can be visualized as a \textit{necklace with beads} where the beads represent spherical Type 1 actions and the strings correspond to admissible compatibilities among those beads (see Section 3 for details). 

Each invariant curve orbit arising in an $\alpha$-compatibility (in a canonical decomposition of $h$) contributes one length and one twist parameter to $\Fix(\langle h \rangle )$, while the addition of each $g_1$-permutation component contributes $\dim(\Teich(S_{g_1,1}))+1$ Fenchel-Nielsen coordinates. Moreover, each length parameter contributed by an $\alpha$-compatibility is bounded above by a positive constant determined by $h$. Thus, a canonical decomposition of $h$ yields a canonical decomposition of $\Fix(H)$ as a product of two-dimensional strips (subspaces) at least one of which is of bounded width. (For details, we refer the reader to Section~\ref{sec:main}.) 


\begin{theorem*}
\label{intro:main1}
Let $H <\Mod(S_g)$ be a cyclic subgroup with $|H| \geq 3$. Suppose that $H$ decomposes into $k$ spherical Type 1 components that were pasted across $p$ many $\alpha$-compatibilities (where $\alpha<n$), $q$ self $\alpha$-compatibilities, and with the addition of a $g_1$-permutation component and the deletion of a $g_2$-permutation component (where $g_2 \leq g_1+q$). Then
$$\Fix(H) \approx \left( \prod_{j=1}^{N}((0,\ell_{j}(H))  \times \mathbb{R}) \right) \times \left( \mathbb{R}_+ \times \mathbb{R}\right )^{2(g_1-g_2)-1},$$
\noindent where $N= 3k+g_1-g_2-2p+q-4 >0$, $\ell_{j}(H) <\infty$ is a positive constant determined by $H$ such that the factors involving $(0, \ell_{j}(H))$ and $\mathbb{R}_+$ are length parameters arising from $\alpha$-compatibilities and the addition (or deletion) of permutation components, respectively, and the remaining factors in each summand being the corresponding twist parameters arising from these constructions (with the implicit understanding that when $g_1-g_2$ is zero, then the last term will disappear). 
\end{theorem*}
\noindent It is important to note that in general $g_2 >0$ in Theorem~\ref{intro:main1}, as the deletion of permutation components are necessary to build arbitrary finite-order mapping classes (see Example~\ref{eg: irr_type2_necklace}).

Let $\sys : \Teich(S_g) \to \mathbb{R}^+$ be the systole function. The systole function has been extensively studied. Using a variational approach initiated by Schmutz~\cite{PS2}, Akrout~\cite{HA1} proved that $\sys$ is a topological Morse function. Moreover, its critical points, and in particular, its local maxima have been characterized~\cite{BC0,PS2}. It is easy to see that by an immediate application of the Gauss-Bonnet Theorem for closed hyperbolic surfaces, the restriction $\sys : \Fix(H) \to \mathbb{R}^+$ is bounded above. When $H= \langle h \rangle$, better bounds for $\sys$ are known~\cite{BBCT}, and for the specific case when $\O_h$ is a triangular surface, a method of explicitly computing $\sys$ was described in~\cite{HK}. Applying Theorem~\ref{intro:main1}, in Section~\ref{sec:injrad-Fix}, we derive an easily computable upper bound for $\sys$ which turn out to be better than the bounds provided by~\cite{BBCT} in many cases. 

\begin{theorem*}
\label{intro:inj_bd}
Let $H = \langle h_1,\ldots h_s \rangle$ be a finite subgroup of $\Mod(S_g)$ such that $|h_i| \geq 3$ for some $1 \leq i \leq s$. Then the restriction $\sys : \Fix(H) \to \mathbb{R}^+$ of the systole functional is bounded above by a global constant $\U(H)$ that is determined by the injectivity radii of points in the compatible orbits under a suitable decomposition of the $h_i$s into spherical Type 1 components as in the hypothesis of Theorem~\ref{intro:main1}. Consequently, the injectivity radii of the structures in $\Fix(H)$ are bounded by $\frac{1}{2}\U(H)$.
\end{theorem*}

\noindent The upper bound $\U(H)$ in Theorem~\ref{intro:inj_bd} is optimal in the sense that for each $\epsilon>0,$ there exist a point $P\in S_{g}$ and a structure $X \in \Fix(H)$ such that the injectivity radius $\inj_P(X)$ at the point $P$ is $>\U(H)-\epsilon$ (see Example~\ref{eg:inj_bd_real1}). 

Let $\mathfrak{g}_{WP}$ be the Weil-Petersson metric on $\Teich(S_g)$. By viewing $\Fix(H)$ as a K\"{a}hler submanifold of $(\Teich(S_g),\mathfrak{g}_{WP})$~\cite{SK2,MW1},  and appealing to the symplectic non-squeezing theorem due to Gromov~\cite{MG} from symplectic geometry, we conclude the following.

\begin{corollary*}
Given a finite subgroup $H < \Mod(S_g)$ with at least one generator of order $\geq 3$, $\Fix(H)$ is not symplectomorphic to the Euclidean space of the same dimension.
\end{corollary*}

Finally in Section \ref{sec:classical_results}, taking inspiration from the theory developed in this paper and~\cite{PRS}, we recover some classical results. In particular, we provide alternative proofs for the following well known results due to Harvey~\cite{H1,HM1} and Gilman~\cite{JG3}.  
\begin{corollary*}
 Let $H = \langle h \rangle$ be an arbitrary $C_n$-action on $S_g$ such that $\O_h$ has $c$ cone points. Then:
\begin{enumerate}[(i)]
\item (Harvey) $\dim(\Fix(H)) = 6g_0(h) + 2c - 6$, and
\item (Gilman) $h$ is irreducible if, and only if $g_0(h) = 0$ and $c=3$.
\end{enumerate}
\end{corollary*}
\noindent We conclude the paper by sketching a purely topological proof of the Nielsen realization theorem for cyclic groups~\cite{WF2,WF1,SK1,JN1}, which asserts that every finite-order element in $\Mod(S_g)$ has a representative in $\text{Homeo}^+(S_g)$ of the same order. 

\section{Preliminaries}\label{prel}
Let $H=\langle h\rangle \leq \Mod(S_g)$ be a cyclic subgroup of order $n$. Throughout this section we will fix an order $n$ representative of $h$ in $\text{Homeo}^{+}(S_g)$ i.e. a $C_n$-action on $S_g$.  

A $C_n$-action $D$ on $S_g$ induces a branched covering $S_g \to S_g/C_n,$ where the quotient orbifold $\O_D : =S_g/C_n$ has signature $(g_0;\,n_1,\dots,n_{\ell})$ (see~\cite{FM,WT}). From orbifold covering space theory, we obtain an exact sequence
$$ 1 \to \pi_1(S_g) \to \pi^{\text{orb}}_1(\O_D) \stackrel{\rho}{\to} C_n \to 1,$$ where $\pi^{\text{orb}}_1(\O_D)$ is a Fuchsian group~\cite{SK} given by the presentation
\begin{gather*}
\langle \alpha_1,\dots,\alpha_{\ell}, x_1,y_1,\dots,x_{g_0},y_{g_0}\,|\,
   \alpha_1^{n_1}=\cdots =\alpha_{\ell}^{n_l}=1,\, \prod_{i=1}^{\ell}\alpha_i=\prod_{j=1}^{g_0}[x_j,y_j]\,\rangle.
\end{gather*}
The epimorphism $\pi^{\text{orb}}_1(\O_D)\xrightarrow{\rho}C_n$ is called the surface kernel map~\cite{WH}, and has the form $\rho(\alpha_i) = t^{(n/n_i)c_i}$, for $1 \leq i \leq \ell$, where $C_n = \langle t \rangle$ and $\gcd(c_i,n_i) = 1$. The map $\rho$ is often described by a $(g_0;\,n_1,\dots,n_{\ell})$-generating vector~\cite{SAB,JG4}.  From a geometric viewpoint, a cone point of order $n_i$ lifts to an orbit of size $n/n_i$ on $S_g$, and the local rotation induced by $D$ around the points in the orbit is given by $2 \pi c_i^{-1}/n_i$, where $c_i c_i^{-1} \equiv 1 \pmod{n_i}$.  (For more details on the theory of finite group actions on surfaces, we refer the reader to~\cite{TB,BA1,HZ}.) 

Putting together the notions of orbifold signature and the generating vector, we can obtain a combinatorial encoding of the conjugacy class of a cyclic action. 

\begin{definition}\label{defn:data_set}
A \textit{data set of degree $n$} is a tuple
$$
D = (n,g_0, r; (c_1,n_1), (c_2,n_2),\ldots, (c_{\ell},n_{\ell})),
$$
where $n\geq 1$, $ g_0 \geq 0$, and $0 \leq r \leq n-1$ are integers, and each $c_i$ is a residue class modulo $n_i$ such that:
\begin{enumerate}[(i)]
\item $r > 0$ if, and only if $\ell = 0$, and when $r >0$, we have $\gcd(r,n) = 1$, 
\item each $n_i\mid n$,
\item for each $i$, $\gcd(c_i,n_i) = 1$, 
\item for each $i$, $\lcm(n_1,\ldots \widehat{n_i}, \ldots,n_{\ell}) = \lcm(n_1,\ldots,n_{\ell})$, and $\lcm(n_1,\ldots,n_{\ell}) = n$, if $g_0 = 0$,  and
\item $\displaystyle \sum_{j=1}^{\ell} \frac{n}{n_j}c_j \equiv 0\pmod{n}$.
\end{enumerate}
The number $g$ determined by the Riemann-Hurwitz equation
\begin{equation*}\label{eqn:riemann_hurwitz}
\frac{2-2g}{n} = 2-2g_0 + \sum_{j=1}^{\ell} \left(\frac{1}{n_j} - 1 \right) 
\end{equation*}
is called the \emph{genus} of the data set, which we shall denote by $g(D)$. Given a data set $D$ as above, we define $$n(D) : = n, \, g(D) := g, \, r(D) = r, \text{ and }g_0(D) := g_0.$$
The quantity $r(D)$ associated with a data set $D$ will be non-zero if, and only if, $D$ represents a free rotation of $S_{g(D)}$ by $2\pi r(D)/n$. 
\end{definition}

 \noindent The following lemma is a consequence of the classical results in~\cite{WH,JN2}. (For more details, see \cite{ALE,km,KP}.)

\begin{lemma}\label{prop:ds-action}
Data sets of degree $n$ and genus $g$ correspond to the conjugacy classes of $C_n$-actions on $S_g$. 
\end{lemma}

\noindent From here on, we will follow the nomenclature of data sets to describe $C_n$-actions on $S_g$. In the remainder of this section, we will summarize the theory developed in~\cite{PRS} to be used extensively throughout this paper. 

\noindent To begin with, we classify $C_n$-actions on $S_g$ into three broad categories. 

\begin{definition}\label{def:types_of_actions}
Let $D$ be a $C_n$-action on $S_g$. Then $D$ is said to be a: 
\begin{enumerate}[(i)]
\item \textit{rotational action}, if either $r(D) \neq 0$, or $D$ is of the form 
$$(n,g_0;\underbrace{(s,n),(n-s,n),\ldots,(s,n),(n-s,n)}_{k \,pairs}),$$ 
for integers $k \geq 1$ and $0<s\leq n-1$ with $\gcd(s,n)= 1$, and $k=1$, if and only if $n>2$.  
\item \textit{Type 1 action}, if $\ell = 3$, and $n_i = n$ for some $i$.  (Note that this is a special type of quasi-platonic action~\cite{BA}.)
\item \textit{Type 2 action}, if $D$ is neither a rotational nor a Type 1 action.
\end{enumerate}
\end{definition} 

\noindent If $g_0(D) = 0$, then we call $D$ a \textit{spherical} action. The following theorem gives a geometric realization of spherical Type 1 actions.
\begin{theorem}\label{res:1}
For $g \geq 2$, a spherical Type 1 action $$D = ((n,0;(c_1,n_1),(c_2,n_2),(c_3,n))$$ on $S_g$ can be realized explicitly as the rotation $\theta_D = 2\pi c_3^{-1}/n$ of a hyperbolic polygon $\p_D$ with a suitable side-pairing $W(\p_D)$, where $\p_D$ is a hyperbolic  $k(D)$-gon with
$$ \small k(D) := \begin{cases}
2n, & \text { if } n_1,n_2 \neq 2, \text{ and } \\
n, & \text{otherwise, }
\end{cases}$$
and for $0 \leq m\leq n-1$, 
$$ \small
W(\p_D) =
\begin{cases}
\displaystyle  
  \prod_{i=1}^{n} a_{2i-1} a_{2i} \text{ with } a_{2m+1}^{-1}\sim a_{2z}, & \text{if } k(D) = 2n, \text{ and } \\
\displaystyle
 \prod_{i=1}^{n} a_{i} \text{ with } a_{m+1}^{-1}\sim a_{z}, & \text{otherwise,}
\end{cases}$$
where $\displaystyle z \equiv m+qj \pmod{n}, \,q= (n/n_2)c_3^{-1}$, and $j=n_{2}-c_{2}$.
\end{theorem}

\begin{definition}
\label{rem:triv_self_comp}
Let $D = (n,g_0; (c_1,n_1), (c_2,n_2),\ldots, (c_{\ell},n_{\ell}))$ be a $C_n$-action on $S_g$.  For a given $g'\geq 1$, one can obtain a new action  from $D$ by removing cyclically permuted (mutually disjoint) disks around points in an orbit of size $n$, and then attaching $n$ copies of the surface $S_{g',1}$ along the resultant boundary components. The resultant action, which is uniquely determined up to conjugacy, is denoted by $ \l D, g'\r$, where
$$ \l D, g'\r := (n,g_0+g'; (c_1,n_1), (c_2,n_2),\ldots, (c_{\ell},n_{\ell})).$$

\noindent Given an action of type $\l D,g'\r$ for some $g' \geq 1$, one can reverse the construction process described above to recover the action $D$. We denote this reversal process by $\overline{\l D,g'\r}$ (i.e. $\overline{\l D,g'\r}=D$).
\end{definition}

It is easy to see that a construction of type $\l D, g'\r$ (or the \textit{addition of a $g'$-permutation component}) and $\overline{\l D,g'\r}$ (or the \textit{deletion of a $g'$-permutation component}) for some $g'>0,$ can be realized by $g'$ inductively performed constructions of type $\l D, 1\r$ and $\overline{\l D,1\r}$, respectively.  We will now describe a construction of a new $C_n$-action from a pair of existing $C_n$-actions across a pair of compatible orbits of size $m$, where $m$ is a proper divisor of $n$. 
 
\begin{definition}\label{def:comp_pair}
For $i = 1,2$, two actions
$$D_{i}=(n, g_{i,0}; (c_{i,1} , n_{i,1} ),(c_{i,2},n_{i,2}),\ldots,(c_{i,\ell_i},n_{i,\ell_i}))$$
are said to form an $(r,s)$-\textit{compatible pair} $D = \lp D_1,D_2, (r,s) \rp$ if there exists $1 \leq r \leq \ell_1$ and $ 1 \leq s \leq \ell_2$ such that 
\begin{enumerate}[(i)]
\item $n_{1,r} = n_{2,s} = m$, and
\item $c_{1,r}+c_{2,s} \equiv 0 \pmod{m}$. 
\end{enumerate}
The number $1+g(D) - g(D_1) - g(D_2)$ will be denoted by $A(D).$
\end{definition}

\noindent This is a formalization of the $\alpha$-compatibility described in Section~\ref{sec:intro}, where $\alpha = A(D)$. The following lemma provides a combinatorial recipe for constructing a new action from an $(r,s)$-compatible pair of existing actions. 

\begin{lemma} \label{lem:comp_pair}
Given a pair of cyclic actions as in Definition~\ref{def:comp_pair}, we have
\begin{gather*}
 \lp D_1,D_2, (r,s) \rp=(n,g_{1,0}+g_{2,0};(c_{1,1},n_{1,1}),\dots,\widehat{(c_{1,r},n_{1,r})}, \ldots, (c_{1,\ell_1},n_{1,\ell_1}),-\\ 
 (c_{2,1},n_{2,1}),\dots,\widehat{(c_{2,s},n_{2,s})}, \ldots, (c_{2,\ell_2},n_{2,\ell_2})), 
\end{gather*}
where $A\lp D_1,D_2, (r,s)\rp := \frac{n}{n_{1,r}}.$ 
\end{lemma}

\noindent It is always possible to construct a new $C_n$ action from a pair of $C_n$ actions $D_i$ as in Definition~\ref{def:comp_pair} across a pair of orbits of size $n$. 
\begin{definition}
Given actions $D_i$ as in Definition~\ref{def:comp_pair}, the action 
\begin{gather*}
\lp D_1,D_2\rp :=(n,g_{1,0}+g_{2,0};(c_{1,1},n_{1,1}),\ldots, (c_{1,\ell_1},n_{1,\ell_1}),-\\
 (c_{2,1},n_{2,1}), \ldots, (c_{2,\ell_2},n_{2,\ell_2})), 
\end{gather*}
where $g(\lp D_1,D_2\rp) = g(D_1)+g(D_2)+n-1$ and $A\lp D_1, D_2 \rp := n$, is said to be a \textit{full compatibility} between the actions $D_i$. 
\end{definition}

\noindent A pair of compatible orbits of the same action on a surface can also be identified to build a new action.

\begin{definition}\label{def:self_comp_ds}
For $\ell \geq 4$, let $D = (n,g_0; (c_1,n_1), (c_2,n_2),\ldots, (c_{\ell},n_{\ell})),$
be a $C_n$-action. Then $D$ is said yield an $(r,s)$-\textit{self compatible} action $D' = \l D,(r,s)\r$, if there exist $1 \leq r < s \leq \ell$ such that
\begin{enumerate}[(i)]
\item $n_r = n_s = m$, and
\item $\displaystyle c_r+c_s \equiv 0 \pmod{m}$. 
\end{enumerate}
The number $g(D')- g(D)$ will be denoted by $A(D')$.
\end{definition}

\noindent Note that this is the self $\alpha$-compatibility described in Section~\ref{sec:intro}, where $\alpha = A(D')$.  The following result gives an explicit realization of the $(r,s)$-\textit{self compatible} action yielded by an action $D$ as above.

\begin{lemma} \label{lem:self_comp_ds}
Let $D$ be an $(r,s)$-self compatible $C_n$-action as in Definition~\ref{def:self_comp_ds}. Then
 we have $$\l D,(r,s)\r=(n,g_0+1;(c_{1},n_{1}),\dots,\widehat{(c_{r},n_{r})}, \ldots, \widehat{(c_{r},n_{s})}, 
 \dots, (c_{\ell},n_{\ell})),$$ where
 $g(\l D,(r,s)\r) = g(D)+n/n_r.$
\end{lemma}
We conclude this section with the following theorem that provides a realization of arbitrary Type 2 actions on $S_g$.
\begin{theorem}\label{thm:arb-real}
For $g \geq 2$, a Type 2 action on $S_g$ can be constructed from finitely many compatibilities of the following types between spherical Type 1 actions:
\begin{enumerate}[(i)]
\item $\l D, (r,s)\r$,
\item $\l D,g' \r$, $\overline{\l D,g' \r}$,
\item $\lp D_1,D_2, (r,s)\rp$, and
\item $\lp D_1,D_2 \rp$.
\end{enumerate}
\end{theorem}

\section{Decomposing cyclic actions into irreducibles}
\label{sec:cyclic_irreducibles}
In this section, we generalize Theorem~\ref{thm:arb-real} to obtain a topological description of the decomposition of an arbitrary cyclic action into irreducible components. We show that this decomposition can be visualized as a ``necklace with beads", where the beads are the irreducible components, and strings that connect a pair of beads symbolize the compatibility between them. We will now present an example that captures this idea.
\begin{example}
\label{eg: irr_type2_necklace}
Consider the spherical Type 1 actions $D_1=(42,0;(2,21), \linebreak(19,42),(19,42))$, $D_2=(42,0;(5,6),(13,21),(23,42))$, $D_3 = (42,0;(1,14),\linebreak(8,21),(23,42))$, $D_4=(42,0;(1,6),(11,21),(13,42))$, $D_5=(42,0;(13,14), \linebreak(10,21),(25,42))$, and $D_6 = (42,0;(19,21),(17,42),(29,42))$. The compatibilities $\lp D_1, D_2, (3,3)\rp$, $\lp D_2, D_3, (2,2) \rp$, $\lp D_3,D_4 \rp$, $\lp D_4, D_5, (2,2) \rp$, and $\lp D_5, D_6, (3,2) \rp$, together realize the action \\
$D' = (42,0;(2,21),(19,42),(5,6),(23,42),(1,14),(1,6),(13,42),(13,14),\linebreak(19,21),(29,42))$
on $S_{155}$.
\\ 
\noindent A visual interpretation of this realization is shown in Figure~\ref{fig:linear_chain} below, where the number of lines connecting $D_i$ to $D_j$ are the sizes of the compatible orbits. (Note that the number 42 refers to the  number of lines connecting $D3$ to $D_4$.)
 \begin{figure}[H]
\labellist
\small
\pinlabel $D_1$ at 77 65
\pinlabel $D_2$ at 280 65
\pinlabel $D_3$ at 483 65
\pinlabel $42$ at 584 110
\pinlabel $D_4$ at 685 65
\pinlabel $D_5$ at 888 65
\pinlabel $D_6$ at 1091 65
\endlabellist
\centering
\includegraphics[width = 70 ex]{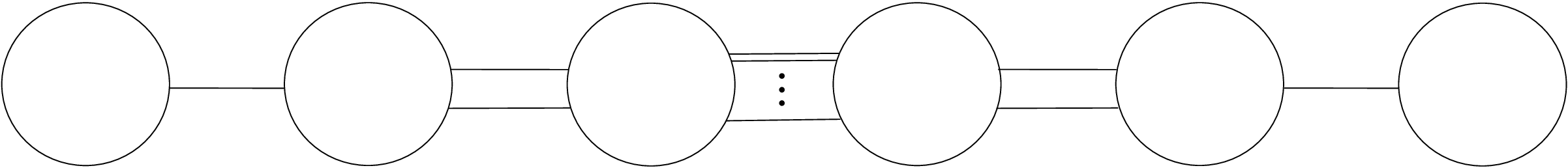}
\caption{A visualization of the action $D'$.}
\label{fig:linear_chain}
\end{figure}
\noindent
\end{example}

\begin{remark}
While we realize new actions from successive compatibilities across actions (represented by data sets), for simplicity, we assume from here on that the original indexing of the pairs (that correspond to cone points) in the data sets remains unaltered during the entire process. 
\end{remark}

\noindent Fixing the notation $\lp D_1, D_2, (0,0) \rp := \lp D_1,D_2 \rp$, we formalize this idea in the following definition.

\begin{definition}
\label{def:linear_chain}
For $1 \leq i \leq k$, let $D_i$ be a collection of cyclic actions of order $n$ on $S_{g_i}$. Then the $D_i$ are said to form a \textit{linear $k$-chain $T=(D_1,\ldots,D_k)$} if the following conditions hold. 
\begin{enumerate}[(i)]
\item Each $D_i$ is a spherical Type 1 action.
\item For $1 \leq i \leq k-1$, there exists non-negative integers $r_i$ and $s_i$ such that the tuples given by $$D_1' = \lp D_1, D_2, (r_1,s_1) \rp \text{ and }
D_j' = \lp D_j', D_{j+1}, (r_j,s_j) \rp, \text{ for } 2 \leq j \leq k-1,$$ are well-defined data sets.
\end{enumerate}
\noindent If in addition to conditions (i)-(iii), there exist non-negative integers $\tilde{r}_k$ and $\tilde{s}_k$ such that the tuple
$D_k' =  \lp D_{k-1}',D_1', (\tilde{r}_k,\tilde{s}_k) \rp$ is also a well-defined data set, then $T$ is said to be a \textit{closed $k$-chain}.
\end{definition}
\noindent Given a $k$-chain $T$ as above, we fix the following notation.
\begin{enumerate}[(a)]
\item $T_{i,j}:=(D_i,D_{i+1},\ldots,D_j)$, if $T$ is not closed.
\item $\mathfrak{C}(T) := \begin{cases}
((r_1,s_1),\ldots,(r_{k-1},s_{k-1}),(\tilde{r}_k,\tilde{s}_k)), & \text{if } T \text{ is closed, and} \\
((r_1,s_1),\ldots,(r_{k-1},s_{k-1})), & \text{otherwise.}
\end{cases}$
\item $D_T := 
\begin{cases}
D_k', & \text{if T is closed, and}\\
D_{k-1}', & \text{otherwise.}
\end{cases}$
\item $A_T := 
\begin{cases}
\sum_{i=1}^k A(D_i'), & \text{if T is closed, and}\\
\sum_{i=1}^{k-1} A(D_i'), & \text{otherwise.}
\end{cases}$ 
\item $f(T) := |\{j:(r_j,s_j) = (0,0)\}|$.
\end{enumerate}

\noindent It is implicit in Definition~\ref{def:linear_chain} that for $1 \leq i < j \leq k$, the tuple $T_{i,j}$ forms a linear $(j-i+1)$-chain. 
\begin{example}\label{eg:linear} 
In Example~\ref{eg: irr_type2_necklace} above, $(D_i,D_{i+1}, \ldots, D_j)$, for $1 \leq i 
< j \leq 6$ form linear chains. In particular, for the linear chain $T = (D_1,\ldots,D_6)
$, we have $\mathfrak{C}(T) = ((3,3),(2,2),(0,0),(2,2),(3,2))$.
\end{example}
\begin{example}
\label{eg:generic_necklace}
In Example~\ref{eg: irr_type2_necklace}, we can simultaneously add the self compatibilities $\l D', (1,9)\r$, $\l D', (2,4) \r$, $\l D', (5,8) \r$, and $\l D', (7,10) \r$ to realize the $C_{42}$-action on $S_{162}$ given by $D''=(42,4; (5,6), (1,6)).$ An illustration of this realization is given in Figure~\ref{fig:irr_Type2_necklace} below. 
 \begin{figure}[H]
\labellist
\small
\pinlabel $D_1$ at 77 125
\pinlabel $D_2$ at 280 125
\pinlabel $D_3$ at 483 125
\pinlabel $42$ at 584 170
\pinlabel $D_4$ at 685 125
\pinlabel $D_5$ at 888 125
\pinlabel $D_6$ at 1091 125
\endlabellist
\centering
\includegraphics[width = 70 ex]{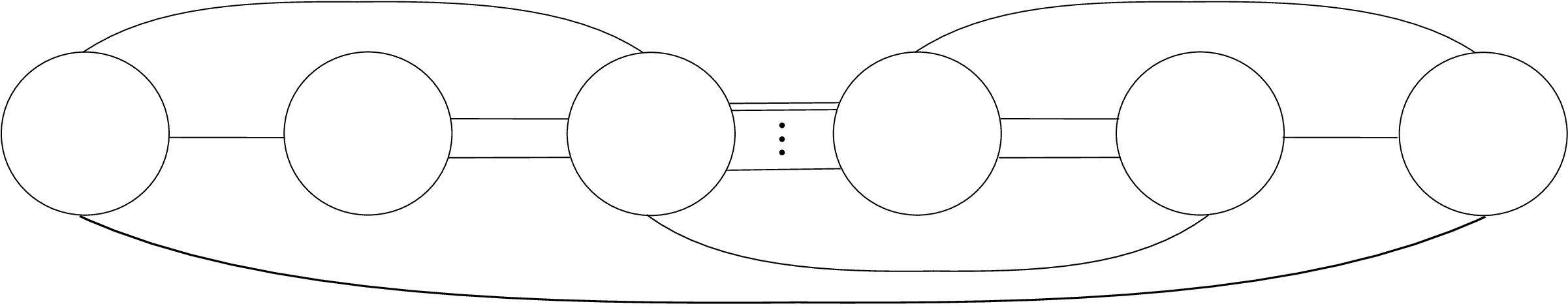}
\caption{A visualization of the action $D''$.}
\label{fig:irr_Type2_necklace}
\end{figure}
\noindent Furthermore, we delete a $3$-permutation component to obtain a realization of the action $D=(42,1; (5,6), (1,6))$ on $S_{36}$.
\end{example}

\noindent This leads us to the following definition for which we fix the notation that $\l D, 0 \r := D$ and $\overline{\l D, 0 \r} := D.$
\begin{definition}
\label{def: compatibility-gen}
Let $T'=(D_1,\ldots,D_k)$ be a linear $k$-chain as in Definition~\ref{def:linear_chain}. Then for integers $1 \leq x_i < y_i \leq k$, $g' \geq 0$, and $0 \leq g'' \leq g'+m$,  the tuple 
$$\N := (T';((x_1,y_1),\ldots,(x_m,y_m)); (g',g''))$$ 
is said to form a \textit{necklace with $k$ beads} if it satisfies the following conditions.
\begin{enumerate} [(i)]
\item If $m>1$, then for $1 \leq j \leq m$, there exists pairs of non-negative integers $(r_j',s_j')$ and the tuples $D_T^j$ given by
$$D_T^1 := \l D_{T'}, (r_1',s_1') \r \text{ and } D_T^{j} := \l D_T^{j-1}, (r_j',s_j') \r,  \text{ for }2 \leq j \leq m,$$ are well-defined data sets. 
\item  For $1 \leq j \leq m$, the tuples $T_{x_j,y_j} = (D_{x_j},D_{x_j+1},\ldots,D_{y_j})$ form closed chains satisfying $A_{T_{x_j,y_j}} = A(D_T^j)$.
\item The tuples $D_{\N}':=\l D_T^{m-1},g' \r$ and $D_{\N} :=\overline{\l D_{\N}', g''\r}$ are both well-defined data sets. 
\end{enumerate}
\end{definition}
\noindent Given a necklace $\N$ as in Definition~\ref{def: compatibility-gen}, we define
$$T_{\N}:= T' \text{ and } \mathfrak{C}(\N) = (\mathfrak{C}(T_{\N}); ((r_1',s_1'),\ldots,(r_m',s_m'))).$$
\noindent It follows by definition that if we replace the $(g',g'')$ with a pair $(g'+p,g''+p)$, where $p$ is a natural number, then the necklace remains unchanged. So for the case when $g' = g''$, we simply omit the pair $(g',g'')$. Moreover, we allow $m = 0$ in a necklace $\N$, in which case, we simply write $\N := (T_{\N}; (g',g'')).$ 

\begin{example}
\label{eg:specific_necklace}
Going back to Example~\ref{eg:generic_necklace},  we see that the action $D$ is realized as a necklace with $6$ beads $$\N = ((D_1,\ldots,D_6);((1,3),(1,6),(3,5),(4,6));(0,3)),$$
where $$\mathfrak{C}(\N) = (((3,3),(2,2),(0,0),(2,2),(3,2));((3,3),(1,1),(1,1),(3,3))).$$ It is interesting to note that the subnecklaces $((D_1,D_2,D_3);((1,3));(0,1))$ and $((D_4,D_5,D_6);((4,6));(0,1))$ are spherical Type 2 actions.
\end{example}
\noindent We will now show that an arbitrary cyclic action can be realized as a necklace, as described in Definition~\ref{def: compatibility-gen}. 

\begin{theorem}
\label{thm:arb_cyc_real}
Given an arbitrary cyclic action $D$ on $S_g$ with $n(D) \geq 3$, there exists a necklace $\N$ with $k$ beads, for some $k \geq 0$, such that $D_{\N} = D$.
\end{theorem}

\begin{proof}
We first consider the case when $D$ is a rotational action. If $D$ is a free rotation of $S_g$ by $2\pi r/n$, then it is of the form $D = (n,\frac{g-1}{n}+1,r;)$, which is realized by the necklace
by $\N = ((D_1,D_1^{-1}); ((1,2),(1,2));)$, where 
$$D_1 = 
\begin{cases} 
(n,0;(r,n),(r,n),(r(n-2),n)), & \text{if } n \text{ is odd, and} \\
(n,0;(r,n),(r,n),(r(\frac{n-2}{2}),\frac{n}{2})), & \text{if } n \text{ is even,}
\end{cases}$$

\noindent $D_1^{-1}$ represents the inverse of the action $D_1$, and $\mathfrak{C}(T_{\N}) = 
((1,1),(2,2),(3,3))$. Along similar lines, we can see that when $D$ is a non-free rotation of $S_g$ by $2\pi r/n$, it is of the form $D_1 = (n,\frac{g}{n};(r,n),(n-r,n);)$, and so $D = D_{\N}$, where $\N = ((D_1,D_1^{-1}); ((1,2));)$ with $D$ taken as above and $\mathfrak{C}(T_{\N}) = 
((1,1),(2,2))$.


Now let $D$ be an arbitrary non-rotational action (with $n(D)\geq 3$). By an inductive application of Theorem~\ref{thm:arb-real}, we can decompose $D$ into finitely many irreducible spherical components $D_1,\ldots,D_k$ with finitely many compatibilities between them of types $\l D', (r,s)\r$, $\l D',g' \r$, $\overline{\l D,g' \r}$, $\lp D_1',D_2', (r,s)\rp$, and $\lp D_1',D_2' \rp$. By a suitable rearrangement and relabeling of the $D_i$, it can be seen that the $D_i$ form a linear $k$-chain $T' = (D_1,\ldots,D_k)$, where each compatibility between $D_i$ and $D_{i+1}$ is chosen among the compatibilities of types $\l D', (r,s)\r$, $\lp D_1',D_2', (r,s)\rp$, and $\lp D_1',D_2' \rp$, arising from the earlier decomposition. (This is possible because any decomposition of $D$ into irreducible Type 1 actions by virtue of Theorem~\ref{thm:arb-real} would yield a maximal reduction system whose maximality would be contradicted by the absence of such a rearrangement.) Any remaining compatibilities (of types $\l D', (r,s)\r$, $\lp D_1',D_2', (r,s)\rp$, and $\lp D_1',D_2' \rp$) in the decomposition of $D$ can now be perceived as series of self-compatibilties on the chain $T'$ between  specific pairs of its components $D_i$. Let the collections of all such pairs of indices (corresponding to pairs of the $D_i$ involved in self-compatibilities) be $(x_1,y_1), \ldots,(x_m,y_m)$. 

Finally, any compatibities that remain in the original decomposition will be of types $\l D',g' \r$ and $\overline{\l D,g' \r}$. Thus, assuming that there $g'$ compatibilities of type $\l D', (r,s)\r$ and $g''$ of type $\overline{\l D,g' \r}$, we conclude that $D = D_{\N}$, where 
$$\N := (T';((x_1,y_1),\ldots,(x_m,y_m)); (g',g'')).$$
%
%
%
%
%
\end{proof}

\begin{remark}
It is important to note that given an action $D$, there could exist two distinct necklaces $\N_1$ and $\N_2$ such that $D_{\N_1} = D = D_{\N_2}$. For example, consider the action $D=(5,1;(1,5),(2,5),(2,5))$ on $S_2$. This can be realized by the necklace $\N_1 = ((D');;(1,0))$, where $D' = (5,0;(1,5),(2,5),(2,5))$. Alternatively, $D_{\N_2} = D$, for $\N_2 = ((D_1,D_2,D');((1,3));),$ where $D_1 = (5,0;(1,5),(1,5),(3,5)),$ and  $D_2 = (5,0;(2,5),(4,5),(4,5))$. 
\end{remark}

\section{Structures realizing compatibilities}
\label{sec:str_real_comps}
In this section, we classify the structures that realize the individual components and compatibilities that constitute a necklace, as described in Definition~\ref{def: compatibility-gen}. We begin by describing the structures that realize spherical Type 1 actions, which form the beads of the necklace. 

\subsection{Spherical Type 1 actions} In this subsection, we show that the structure $\p_D$ (described in Theorem~\ref{res:1}) that realizes a Type 1 action $D$ is unique. 

\begin{proposition}\label{thm:irr_type1}
If $D$ is a spherical Type 1 action, then $\text{Fix}(\langle D \rangle) = \{\p_D\}$. 
\end{proposition}

\begin{proof}
First consider the case when $n_i=2$ for some $i$. Then $D$ can be realized as a rotation of the regular hyperbolic $n$-gon $\p_D$ (as in Theorem~\ref{res:1}), with all interior angles equals to $2\pi/n_2$. It follows from basic hyperbolic trigonometry that such a hyperbolic polygon is unique, which proves the result for this case.

When $n_1,n_2\neq 2$, $\p_D$ is a semi-regular hyperbolic $2n$-gon with side length $\ell$, and alternate interior angles of measure $2\pi/n_1$ and $2\pi/n_2$, respectively. Let $\{P_0,\dots,P_{2n-1}\}$ be the vertices of $\p_D$ and $O$ denotes the fixed point at the center, as shown in Figure~\ref{fig:s14_fixpt} below. 
\begin{figure}[H]
\labellist
\small
\pinlabel $P_i$ at 165 8
\pinlabel $P_{i+1}$ at 245 8
\pinlabel $P_{i+2}$ at 315 43
\pinlabel $O$ at 180 200
\endlabellist
\centering
\includegraphics[width = 30 ex]{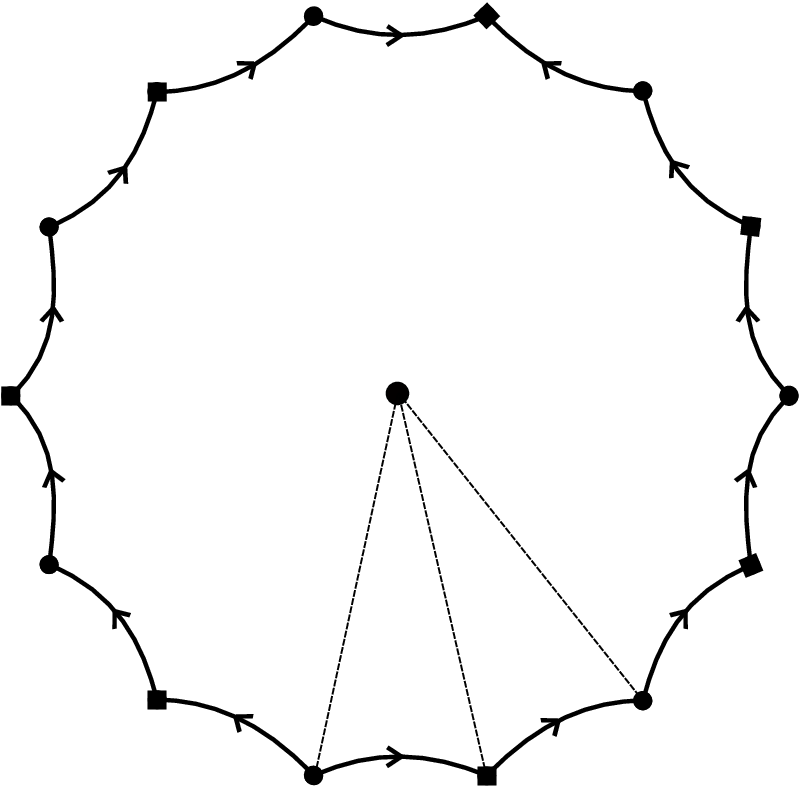}
\caption{The polygon $\p_D$ for a $C_7$-action on $S_3$.}
\label{fig:s14_fixpt}
\end{figure}
As the rotation of $\p_D$ by $\theta_D$ is an isometry, it follows that $|OP_i| = |OP_{i+2}|$, for all $i$.  Hence, the hyperbolic $SSS$ congruence implies that the triangles $P_iOP_{i+1}$ are mutually congruent to each other, with $\angle P_iOP_{i+1}=\pi/n$, $\angle OP_i P_{i+1}=2\pi/n_1$, and $\angle OP_{i+1} P_i=2\pi/n_2$. Thus, $\p_D$ is uniquely determined, and the assertion follows.
\end{proof}

\begin{remark}
\label{rem:fixD}
Let $D$ be a reducible action on $S_g$, and let $\C$ be a maximal reduction system for $D$. By extending $\C$ to a pants decomposition $P$ of $S_g$, we see that $\dim (\Fix ( \langle D \rangle )) \geq 2|\C| > 0$. Conversely, suppose that $\dim(\Fix(\langle D \rangle)) > 0$, we can reverse the above argument to show that $D$ is reducible. 
\end{remark}

\noindent The following corollary is immediate from Proposition~\ref{thm:irr_type1} and Remark~\ref{rem:fixD}.

\begin{corollary}
A spherical Type 1 action $D$ is irreducible. 
\end{corollary}

\subsection{Compatibilities of type $\lp D_1,D_2,(r,s) \rp$ and $\l D, (r,s) \r$} Consider an irreducible Type 1 action $D$ on $S_g$, and a $D$-orbit of size $k.$  Removing $k$ mutually disjoint cyclically permuted (by the action of $D$) discs around the points in this orbit, we obtain a homeomorphic copy of $S_{g,k}$ with a homeomorphism $\hat{D}$ induced by $D$, which cyclically permutes the components of $\partial S_{g,k}$. 
Note that $\Teich(S_g)$ can be viewed as a subspace of $\Teich(S_{g,k})$ in the following manner. The Fenchel-Nielsen coordinates of an arbitrary structure $\xi \in \Teich(S_g)$ are given by $\xi= \prod_i^{3g-3} (\ell_i, \theta_i),$
where the pair $(\ell_i, \theta_i)$ denote the length and twist parameters contributed by the $i$-th curve of a pants decomposition $P$ of $S_g$ where $i=1,\ldots, 3g-3$. $P$ can always be extended to a pants decomposition $\hat{P}$ of $S_{g,k}$ where the first $3g-3$ non-boundary curves of $\hat{P}$ belong to $P$. As there are $3g-3+k$ non-boundary curves in $\hat{P}$, an arbitrary $\hat{\xi} \in \Teich(S_{g,k})$ can be decomposed as 
{\small $$\hat{\xi}= \prod_i^{3g-3+k} (\ell_i, \theta_i) \times \prod_{j=1}^k \ell_{b_j},$$}
where $\ell_{b_j}$ denotes the length parameter of the $j$-th boundary component (for $j=1,\ldots ,k$) of $S_{g,k}.$ 

In light of the above decomposition of $\hat{\xi}$, two natural questions that arise are: ``Does there exist an endomorphism $\hat{D}_{\#}: \Teich(S_{g,k}) \to \Teich(S_{g,k})$ such that $\hat{D_{\#}}|_{\Teich(S_g)}=D_{\#}$? Moreover, is $\hat{D}_{\#}|_{\Teich(S_{g,k})\setminus \Teich(S_g)}$ a permutation of the coordinates?" We will show shortly that these questions do not always have positive answers. Consider the decomposition
$\Teich(S_{g,k})\approx \mathcal{T}_{NB} \times \mathbb{R}_{+}^k,$ where
$$\mathcal{T}_{NB} = \{\prod_i^{3g-3+k} (\ell_i, \theta_i)\} \text{ (NB refers to ``non-boundary") and }
\mathbb{R}_{+}^k \approx \{\prod_{j=1}^k \ell_{b_j}\}.$$ 
The action of $D$ implies that $\hat{D}_{\#}$, if it exists, should preserve the above decomposition of $\Teich(S_{g,k})$, and furthermore, $\hat{D}_{\#}\left(\prod_{j=1}^k \ell_{b_j}\right)= \prod_{j=1}^k \ell_{b_{\sigma_k(j)}}$ where $\sigma_k=(12\ldots k).$ The following result shows that $\hat{D}_{\#}$ is completely determined by $D_{\#}$ if, and only if, $k$ is a proper divisor of $n$.

\begin{proposition}\label{Thm:induced_action}
Let $D$ be a spherical Type 1 action on $S_g$ of order $n$ with a $D$-orbit  of size $k$.  Then $D_{\#}$ never extends to an endomorphism of $\Teich(S_{g,k})$, which induces an order $n$ permutation of the coordinates of $\Teich(S_{g,k})\setminus \Teich(S_g)$. In particular, the extended action $\hat{D}_{\#}$ is completely determined by $D_{\#}$ if, and only if, $k$ is a proper divisor of $n$.
\end{proposition} 

\begin{proof}
As $D$ is an spherical Type 1 action, we may assume that there exists a pants decomposition $P$ of $S_g$ with $s$ separating curves $\alpha_1,\ldots,\alpha_{s}$ and $r$ non-separating curves $\beta_1,\ldots, \beta_{r}$ such that for each $1 \leq i \leq s-1$, there exist $1 \leq j \leq s$ ($j \neq i$) and $1 <M_{ij} <n$ with $D^{M_{ij}}(\alpha_{i})= \alpha_{j}$. Similarly, for each $1 \leq i \leq r-1$, there exist $1 \leq j \leq r$ ($j \neq i$) and $1 <N_{ij} <n$ such that $D^{N_{ij}}(\beta_{i})= \beta_{j}$. Without loss of generality, we may assume that $D^{N_{1,r}}(\beta_1) = \beta_r$. 

In order that $D_{\#}$ extends to an endormorphism of $\Teich(S_{g,k})$, $P$ should extend to a pants decomposition $\hat{P}$ of $S_{g,k}$ as in the discussion above, with $k$ new non-boundary curves $\gamma_1,\ldots,\gamma_k$ and $k$ boundary curves $\gamma_1',\ldots,\gamma_k'$ such that $\hat{D}(\gamma_i')= \gamma_{i+1}'$, for each $i.$ We may assume that $\gamma_1$ is a nonseparating curve isotopic to $\beta_r$ in $S_g$, and thus $\hat{D}^M(\gamma_1)=\beta_{1}$ (since $D^M(\gamma_1)=\beta_{1}$), and the isotopy class of $\beta_1$ remain unaltered in $S_{g,k}$, as illustrated in Figure \ref{fig:c14_action_on_s3} below. In the case when $k = n$, it is apparent that the curve $\sum_{i} \gamma_i' \in H_1(S_{g,k})$ (indicated by the dotted curve in the polygon, and the curve $\gamma_2$ in the bounded surface in Figure \ref{fig:c14_action_on_s3} below) is left invariant by the action of $D$.
\begin{figure}[htbp]
\labellist
\small
\pinlabel $\beta_1$ at 480 240
\pinlabel $\gamma_1$ at 615 260
\pinlabel $\beta_r$ at 620 150
\pinlabel $\gamma_2$ at 665 260
\pinlabel $\gamma_3$ at 713 260
\pinlabel $\gamma_4$ at 713 140
\pinlabel $\gamma_1'$ at 792 300
\pinlabel $\gamma_2'$ at 792 230
\pinlabel $\gamma_3'$ at 792 168
\pinlabel $\gamma_4'$ at 792 100
\pinlabel $\sum \gamma_i'$ at 200 305
\endlabellist
\centering
\includegraphics[width = 60 ex]{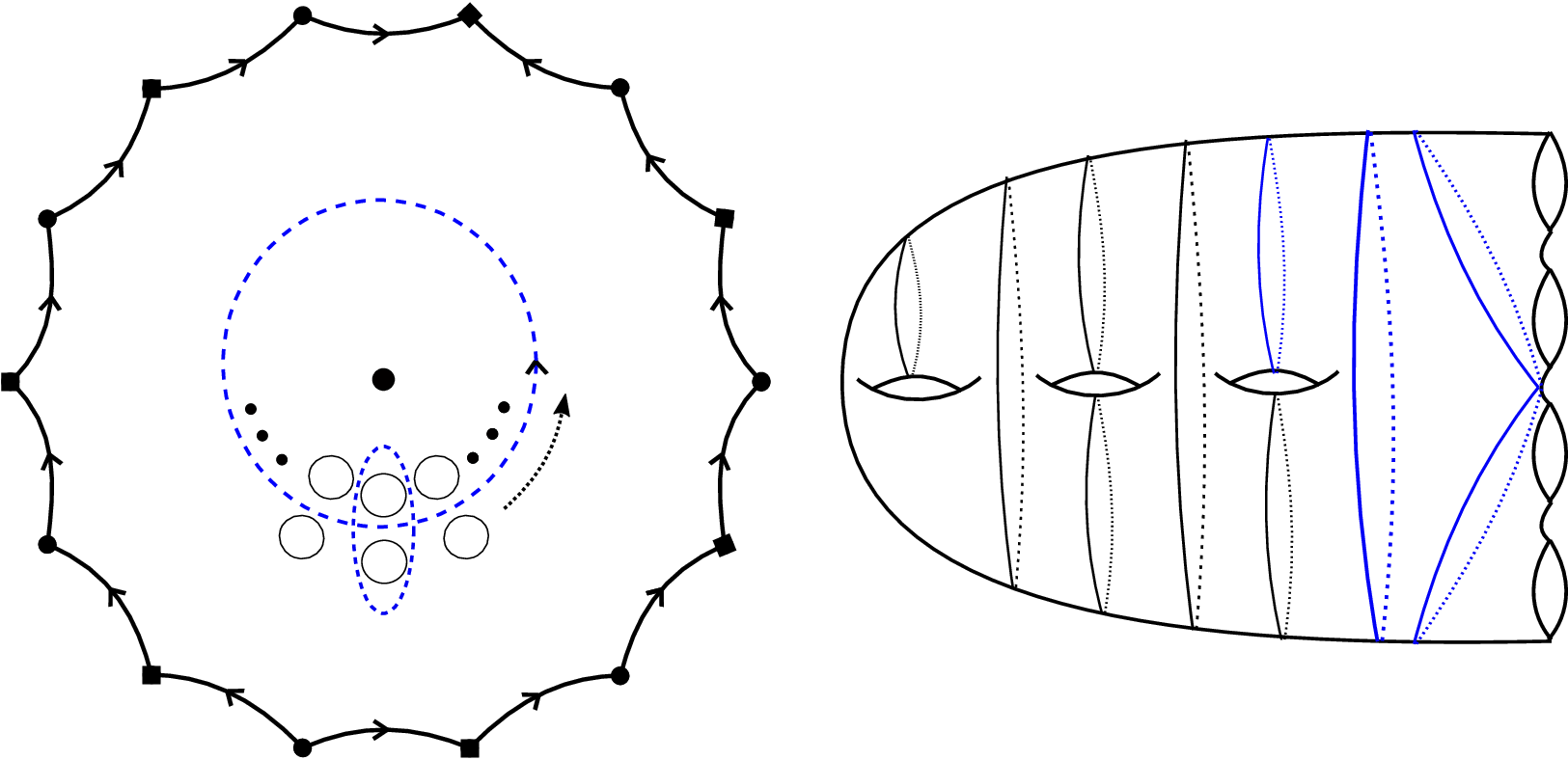}
\caption{Extension of a pants decomposition of $S_g$.}
\label{fig:c14_action_on_s3}
\end{figure}
\noindent Hence, $D$ has to induce an order $n$ rotation 
of the component $S'$ of $\overline{S_g \setminus \gamma_2}$ homeomorphic to $S_{0,k+1}$, which cyclically permutes its $k$ boundary components $\gamma_i'$ and fixes the $k+1$-th boundary component, namely, $\gamma_2$. This obviates the possibility of such an extension in this case, as $D\vert_{S'}$ can never induce an order $k$ permutation of the $\gamma_i$.

Furthermore, it is clear from the structure $\p_D$ that when $k$ is a proper divisor of $n$, then $\gamma_2$ cannot be left invariant by the action of $D$. Consequently, the action of $\hat{D}$ on the $\gamma_i$ is completely determined by the action of $D$ on $P$, and hence the result follows.
\end{proof}

\begin{remark}\label{rem:inj_radius}
Let $(X,\xi)$ be a closed hyperbolic surface with an isometry $D$ of finite order. Let $B_p(r)$ denote the closed disc of radius $r$ centered at any point $p \in X.$ If $D(p)=p$ and $D(B_p(r)) = B_p(r)$ such that $D|_{B_p(r)}$ becomes a rotation about $p$, then $r < \inj_p(X,\xi)$, where $\inj_p(X,\xi)$ denotes the injectivity radius of $(X,\xi)$ at $p$. This is immediate from the fact that the derivative of $D$ at $p$ is a rotation about the origin in $T_pX$, and the exponential map restricted to a normal ball is a radial isometry. 
\end{remark}
 
\noindent The following result describes the structures that realize compatibilities of type $\lp D_1,D_2,(r,s) \rp$.

\begin{corollary}\label{cor:comp_pair}
Let $D= \lp D_1,D_2,(r,s) \rp$, where the $D_i$ are spherical Type 1 actions. 
\begin{enumerate}[(i)]
\item If $(r,s) \neq (0,0)$, then
$$\Fix(\langle D \rangle) \approx \{[\p_{D_1}]\} \times \{[\p_{D_2}]\} \times (0,\ell(D)) \times \mathbb{R},$$
where $\ell(D)$ is a positive constant determined by $D$.
\item If $(r,s) = (0,0)$, then 
$$\Fix(\langle D \rangle)\approx \{[\p_{D_1}]\} \times \{[\p_{D_2}]\} \times \prod_{j=1}^3 \left( (0,\ell_j(D))\times \mathbb{R}\right),$$
where for each $j$, $\ell_j(D)$ is a positive constant determined by $D$.
\end{enumerate} 
\end{corollary}

\begin{proof}
We will only prove (i), as (ii) will follow from a similar argument. By Proposition~\ref{Thm:induced_action}, it is apparent that the action induced by the $D_i$ on $S_{g_i,k}$ is completely determined by the action of $D_i$ on the $S_{g_i}$. So any structure that realizes 
$\lp D_1,D_2,(r,s) \rp$ as an isometry, is uniquely determined by the structures $\p_{D_i}$, and one additional length and twist parameter contributed by the isometric boundary components (cyclically permuted by the $D_i$) of $S_{g_i,k}$. 

Let $\ell$ denote the length of each boundary component of $S_{g_i,k}$.  It remains to show that $\ell < \ell(D)$,  where $\ell(D)$ is a positive constant determined by $D$. To see this, consider the unique hyperbolic surface $(X_i,\xi_{ih})$ (for $i=1,2$) realizing $D_i$ as an isometry. For each $i$, let $\{p_{ij}\}_{1\leq j\leq k} \subset X_i$ be the points in a distinguished compatible $D_i$-orbit of size $k$. Let $B_{ij}(r_i):=B_{p_{ij}}(r_i)$ denote mutually disjoint cyclically permuted disks under $D_i$.  Since $D_i^k(B_{{ij}}(r_i))=B_{{ij}}(r_i),$ it follows from Remark~\ref{rem:inj_radius} that $r_i < \inj_{p_{ij}}(X_i,\xi_{ih})$ (each $p_{ij}$ has the same injectivity radius). Thus the circumference $c_{ij}$ of each $B_{{ij}}(r_i)$ satisfies 
$$c_{ij} =2\pi \sinh (r_i) < 2\pi \sinh (\inj_{p_{ij}}(X_i,\xi_{ih})) =L_i \text{ (say)}.$$
 
Let $L= \min(L_1, L_2),$ and $r_D=\min(\inj_{p_{1j}}(X_1,\xi_{1h}), \inj_{p_{2j}}(X_2,\xi_{2h}))$. Removing $\{B_{ij}(r)\}_{1\leq j\leq k}$ (where $r < r_D$ and the circumference $c(r)$ of $B_{ij}(r)$ satisfies $c(r)< L$) from each $X_i$, and gluing the surfaces $\overline{X_i \setminus \cup_{j} B_{p_{ij}}(r)}$ along their boundary components, we obtain a diffeomorphic copy $X$ of $S_{g_1+g_2+k-1}$ with a $C_n$ action $D$, and a reduction system $\C$ consisting of $k$ nonseparating curves. 
Moreover, $X$ admits a canonical Riemannian metric $\xi$ realizing $D$ as an isometry with each curve of $\C$ having length $c(r)$. By the uniformization theorem, there is a unique hyperbolic metric $\xi_h = e^{f}\xi$ on $X$, also realizing $D$ as an isometry, where $f=f(\xi_1, \xi_2)$ is a smooth real valued function on $X$. The result (i) now follows from the observation that under $\xi_h$, each curve of $\C$ has length $\ell_h= \ell_h(c(r),f) < \ell(D)$ where $\ell(D)= \ell(L, f)$ is a unique constant (as $L,f$ are uniquely determined by $D$).
\end{proof}

\noindent Considering the similarities between the compatibilities $\l D', (r,s) \r$ and \linebreak $\lp D_1,D_2, (r,s) \rp$, it is quite evident that the structures that realize $\l D', (r,s) \r$ should also arise analogously, and so we have the following.
\begin{corollary} \label{cor: self_comp_pair}
Let $D = \l D', (r,s) \r$ be an action of order $n$ on $S_g$. Then, 
$$\Fix(\langle D \rangle) \approx \Fix (\langle D' \rangle) \times (0,\ell(D')) \times \mathbb{R},$$
where $\ell(D')$ is a positive constant determined by $D'$.
\end{corollary}

\subsection{Compatibilities of type $\l D, g_0\r$ and $\overline{\l D, g_0\r}$} Let $D$ be an action of order $n$ on $S_g$. As we saw earlier, an action of type $\l D, g_0\r$ is realized by pasting a permutation component (that cyclically permutes $n$ isometric copies of $S_{g_0,1}$) to the action $D$. As we saw earlier, the action $\l D, g_0\r$ can also be realized iteratively from $g_0$ compatibilities of type $\l D,1 \r$. Besides, the arguments in Theorem~\ref{Thm:induced_action} would imply that each copy of $S_{1,1}$ (that is attached in a $\l D,1 \r$ type construction) contributes $2$ additional length parameters, and $1$ twist parameter. Furthermore, following the arguments in Corollary~\ref{cor:comp_pair}, we can show that one of the length parameters (contributed by $\partial (S_{1,1,})$ is bounded by a positive constant that is determined uniquely by the action on which the permutation component is pasted. Hence, when the compatibility $\l D, g_0 \r$ is completed, a total of $3g_0-1$ length and twist parameters would have been added to the dimension of $\Fix(\langle D \rangle)$, and so we have the following result. 

\begin{corollary}
\label{cor:perm_comp}
Let $D$ be a cyclic action of order $n$ on $S_g$. Suppose that the actions $\l D, g_0 \r$ and $\overline{ \l D, g_1 \r}$ are well-defined, for some $g_0,g_1 \geq 1$. Then 
\begin{enumerate}[(i)]
\item $\displaystyle \Fix(\langle \l D, g_0 \r \rangle)  \approx \Fix(\langle D \rangle ) \times \prod_{i=1}^{g_0}((0, \ell_i^0(D)) \times \mathbb{R}) \times \prod_{i=1}^{2g_0 - 1} (\mathbb{R}_+  \times \mathbb{R}),$ where each $\ell_i^0(D)$ is a positive constant determined by the action $\l D, g_0 \r$.
\item $\displaystyle \Fix(\langle \overline{\l D , g_1 \r} \rangle)  \approx \Fix (\langle D \rangle ) \Large{/}\left(\prod_{i=1}^{g_1}((0, \ell_i^1(D)) \times \mathbb{R}) \times \prod_{i=1}^{2g_1 - 1} (\mathbb{R}_+  \times \mathbb{R})\right),$ where each $\ell_i^1(D)$ is a positive constant determined by the action $\overline{\l D, g_1 \r}$.
\end{enumerate}
\end{corollary}

\subsection{Structures that realize arbitrary actions}
\label{sec:main}
In this subsection, we will piece together the structures detailed in the Section~\ref{sec:str_real_comps} (that realize various kinds of compatibilities) to describe the structures that will realize arbitrary cyclic actions of order $3$ and above. We recall that for such a cyclic action $D$, there exists a necklace 
\[\tag{*} \N = ((D_1,\ldots,D_k); ((x_1,y_1), \ldots, (x_m,y_m)); (g',g''))\]
as in Definition~\ref{def: compatibility-gen}, such that $D_{\N} = D$ (see Proposition~\ref{thm:arb_cyc_real}). Putting together the results in Proposition \ref{Thm:induced_action} and Corollaries \ref{cor:comp_pair} - \ref{cor:perm_comp}, we obtain an explicit decomposition of $\Fix(\langle D\rangle)$ as a product of two-dimensional strips some of which are of bounded width.
\begin{theorem}\label{thm:main}
Let $D$ be a cyclic action  of order $n \geq 3$ on $S_g$, and let $\N$ be a necklace as in $(*)$ such that $D_{\N} = D$. Then $\Fix(\langle D \rangle) \approx M_1/M_2$, where
{ $$ M_1=\prod_{i=1}^{k}\{\p_{D_i}\} \times \prod_{i=1}^{g'+k+2f(T_{\N})+m-2}((0, \ell_i'(D)) \times \mathbb{R}) \times \prod_{i=1}^{2g' - 1} (\mathbb{R}_+  \times \mathbb{R})$$}
and 
$$M_2= \prod_{i=1}^{g''}((0, \ell_i''(D)) \times \mathbb{R}) \times \prod_{i=1}^{2g'' - 1} (\mathbb{R}_+  \times \mathbb{R}),$$
where the $\ell_j'(D)$ and $\ell_j''(D)$ are positive constants determined by $D$. Consequently, 
$$\dim(\Fix(\langle D \rangle) = 6(g'-g'')+2k+4f(T_{\N})+2m-2.$$ 
\end{theorem}

\noindent Since the dimension of manifold $M_2$ (in Theorem~\ref{thm:main}) is determined by the deletion of permutation components in a given realization of $D$, it is crucial to note that it is in general nonempty even when $D$ is irreducible. In fact, it was shown in \cite{PRS} that the deletion of at least one permutation component is required to realize an irreducible Type 2 action.

\begin{example}
For the necklace structure realizing the action $D$ in Example~\ref{eg:generic_necklace}, we see that $k=6$, $m=4$, $f(T_{\N}) =1$, and $(g',g'') = (0,3)$. Consequently, applying Theorem~\ref{thm:main}, we have $\text{Fix}(\langle D \rangle) \approx M_1/M_2$, where 
$$M_1 = \left( \prod_{i=1}^{10} (0, \ell'_i(D)) \right)\times \mathbb{R}^{10} \text{ and } M_2 = \left( \prod_{i=1}^3 (0, \ell''_i(D)) \times \mathbb{R} \right) \times \mathbb{R}_+^5 \times \mathbb{R}^5,$$
and so we have $\dim(\text{Fix}(\langle D \rangle)) = 20-16=4$.
\end{example} 

\section{Applications to the geometry of $\Fix(H)$}
\label{sec:injrad-Fix}
Let $H= \langle D_1, \ldots, D_s \rangle < \Mod(S_g)$ be a finite subgroup, with $n(D_i)\geq 3$, for  at least one $i,$ $1\leq i \leq s$. In this section, we discuss a few applications of Theorem~\ref{thm:main} that help us understand some finer geometric properties of $\Fix(H)$ as a submanifold of $\Teich(S_g)$. 

\subsection{An upper bound of the injectivity radii of structures in $\Fix(H)$}
In this subsection, we begin by deriving an upper bound for the systole (defined below) of hyperbolic structures in $\Fix(H)$, when $H = \langle D\rangle$ represents a cyclic action on $S_g$ of order at least $3$, in terms of the injectivity radii of the irreducible components of $\langle D \rangle$. This, in turn, yields a bound on the injectivity radii of structures in $\Fix(\langle D \rangle)$ and consequently a bound on the injectivity radii of structures in $\Fix(H)$ for any finite subgroup $H$ of $\Mod(S_g)$.

\noindent Let $\inj(X)$ denote the injectivity radius of a Riemannian manifold $X$. 
The \textit{systole} (also known as the \textit{systole length}) of a compact hyperbolic surface $(X,\xi)$, denoted by $\sys(X)$, refers to the length of any of the shortest closed geodesic(s) of $X$. In particular, when $X$ is closed, $\sys(X)=2\inj(X)$, where $\inj(X)$ denotes the injectivity radius of $(X,\xi)$. First, we provide estimates for the injectivity radii of the structures of type $\p_D$ realizing spherical Type 1 actions. For this, we will use the following lemma that applies basic hyperbolic trigonometry. 

\begin{lemma}\label{lem:alt}
Let $ABC$ be a  hyperbolic triangle and let $D$ denote the foot of perpendicular from vertex $A$ to side $BC.$ Then the length of the altitude $AD$ is given by:
$$ \sinh^{-1}{\left( \frac{\sqrt{\cos^2{A}+\cos^2{B}+\cos^2{C}-1+2\cos{A}\cos{B}\cos{C}}}{\sin{A}}\right)} $$
\end{lemma}

\noindent The following lemma provides estimates for injectivity radii at the nontrivial orbit points of a spherical Type 1 action in terms of its data set.

\begin{proposition}\label{prop:inj_spherical_type1}
Let $D=(n,0;(c_1,n_1),(c_2,n_2), (c_3,n_3))$, where $n_3 = n$, be a spherical Type 1 action in $\Mod(S_g)$ realized by a structure $\mathcal{P}_D$ as in Theorem~\ref{res:1}. Let $\inj_P(\p_D)$ be the injectivity radius afforded by the structure $\mathcal{P}_D$ at a point $P \in S_g$.
\begin{enumerate}[(i)]
\item Let $P_0$ be the center of $\p_D$. 
\begin{enumerate}
\item If $n_1,n_2 \neq 2$, then 
$$\inj_{P_0}(\p_D) = \sinh^{-1}{\left( \frac{\sqrt{\displaystyle\sum_{i=1}^3 \cos^2{(\pi/n_i)}-1+2\prod_{i=1}^3\cos{(\pi/n_i)}}}{\sin{(\pi/n)}}\right)}.$$ 
\item If $n_2  =2$, then $\inj_{P_0}(\p_D)$ equals $$\displaystyle \sinh^{-1}{\left( \frac{\sqrt{\cos^2{(2\pi/n)}+2\cos^2{(\pi/n_1)}-1+2 \cos{(2\pi/n)}\cos^2{(\pi/n_1)}}}{\sin{(2\pi/n)}}\right)}.$$
\end{enumerate}
\item Let $P$ be any non-center point which belongs to a $\langle D \rangle$-orbit corresponding to a nontrivial cone point of $\O_D$. Then $\inj_P(\p_D) < \ell(\p_D)$, where $\ell(\p_D)$ denotes the length of an edge of $\mathcal{P}_D$ given by 
$$
\ell(\p_D) = \begin{cases}
 \cosh^{-1} {\left( \frac{\cos^2 (\pi/n_1)+\cos (2\pi/n)}{\sin^2 (\pi/n_1)} \right)}, & \text{if } k(D) = n, \text{ and} \\
\cosh^{-1} {\left( \frac{\cos (\pi/n_1)\cos (\pi/n_2)+\cos (\pi/n)}{\sin (\pi/n_1) \sin (\pi/n_2)} \right)}, & \text{if } k(D)=2n.
\end{cases}$$
\end{enumerate}
\end{proposition}

\begin{proof}
\begin{enumerate}[(i)]
\item By definition, $\inj_P(\p_D)$ is the radius of the biggest hyperbolic disc centered at $P$ that can be isometrically embedded in the polygon $\mathcal{P}_D$. Thus, $\inj_P(\p_D)$ turns out to be the (hyperbolic) inradius of $\mathcal{P}_D$.
 
Here two cases arise depending upon whether $k(D)=n$ or $k(D)=2n.$ We will argue for the case when $k(D)=2n$ as the argument for the other case is similar. Using the geometric realization of $D$ as $\mathcal{P}_D,$ one can see that $\mathcal{P}_D$ can be divided into $2n$ congruent triangles with interior angles $\pi/n$, $\pi/n_1$, and $\pi/n_2,$ sharing a common vertex and interior angle $\pi/n$. Hence, the inradius of $\mathcal{P}_D$ is given by the altitude of any one of these $2n$ hyperbolic triangles joining the vertex with interior angle $\pi/n$ to the opposite side with interior angles $\pi/n_1$ and $\pi/n_2.$ Thus, our assertion now follows from Lemma \ref{lem:alt}.
\item We break our argument into the following two cases.
\begin{enumerate}[\text{Case} 1.]
\item $k(D)=n$ (Without loss of generality, let $n_2=2$).\\
In this case, $\mathcal{P}_D$ is a regular $n$-gon with side length $\ell(\p_D)$ where each vertex and the mid point of each side belong to respective orbits of sizes $n/n_1$ and $n/2$. By Lemma \ref{lem:alt}, we have $\ell(\p_D)= \cosh^{-1} {\left( \frac{\cos^2 (\pi/n_1)+\cos (2\pi/n)}{\sin^2 (\pi/n_1)} \right)}.$
Consider two consecutive vertices $P_1,P_2$ and the mid point $P_{12}$ of the side $P_1P_2$ with respective injectivity radii $r_1$ and $r'.$ From the symmetries of the polygon and the side pairings, it follows directly that radius of the biggest normal ball centered at either of $P_1$ and $P_{12}$ must be strictly bounded above by the side length $\ell(\p_D)$. 

\item $k(D)=2n$ (i.e. $n_1,n_2\geq 3$).\\ 
In this case, $\mathcal{P}_D$ is a hyperbolic $2n$-gon with each side of length $\ell(\p_D)$. By Lemma \ref{lem:alt}, we get \small{$$\ell(\p_D)= \cosh^{-1} {\left( \frac{\cos (\pi/n_1)\cos (\pi/n_2)+\cos (\pi/n)}{\sin (\pi/n_1) \sin (\pi/n_2)} \right)}.$$} Let $P_1,P_2$ be two consecutive vertices with respective injectivity radii $r_1,r_2$, belonging to respective orbits of sizes $n/n_1$ and $n/n_2$. A similar argument as in the previous case, would imply that both $r_1,r_2 <\ell(\p_D)$.

\end{enumerate} 
\end{enumerate}
\end{proof}

\begin{remark}
\label{rem:inj_bound}
It can be easily verified that the maximum injectivity radius of a structure of type $\p_D$ realizing a spherical Type 1 action $D$ must be bounded above by one of the upper bounds mentioned in Proposition \ref{prop:inj_spherical_type1}. In fact, from the symmetries of the polygon, it follows that the maximum radius of a normal ball at any point in the interior of the polygon $\p_D$ must be smaller than the inradius of $\p_D$. Similarly, for any point belonging to the interior of an edge of $\p_D$, the maximum radius of a normal ball must be strictly bounded above by the length of the edge.
\end{remark}

\noindent In view of Proposition~\ref{prop:inj_spherical_type1} and Remark~\ref{rem:inj_bound}, we obtain the following upper bound for the maximum injectivity radius of a structure of type $\p_D$.

\begin{corollary}
\label{cor:bound_inj_pD}
Let $D \in \Mod(S_g)$ be a spherical Type 1 action realized by a structure $\mathcal{P}_D$ as in Theorem~\ref{res:1}. Then
$$M(\inj(\p_D)) \leq \U(\p_D),$$ where $$M(\inj(\p_D)):=\max_{P\in S_g} \{\inj_P(\p_D)\}$$ is the maximum injectivity radius of $S_g$ realized by the structure $\p_D$ and 
$\U(\p_D) := \max\{\inj_{P_0}(\p_D), \ell(\p_D)\}.$
\end{corollary}

\begin{remark}
\label{rem:inj_bound_rscomp}
Let $D \in \Mod(S_g)$ be such that $D= \lp D_1,D_2,(r,s) \rp$, where the $D_i$ are spherical Type 1 actions. Note that a hyperbolic structure that realizes $D$ is obtained by identifying the boundaries of $k$ cyclically permuted isometric discs of radius $R$ across a pair of compatible orbits of size $k$ belonging to the unique geometric realizations of $D_1$ and $D_2$, respectively. The identified boundaries become a family of $k$ closed geodesics in $S_g$ of length $\ell =2\pi \sinh R $ with $R<\min \{R_1,R_2\}$, where $R_i$ denotes the injectivity radius of each point in the corresponding orbit of $D_i$ (using Corollary \ref{cor:comp_pair}). By definition, the systole of $D$, $\sys(D) \leq \ell =2\pi \sinh R$. 

More generally, let $D \in \Mod(S_g)$ such that $D= \lp D_1,D_2,(r,s) \rp$, where one of the $D_i$ is a spherical Type 1 action. Without loss of generality, let us assume $D_1$ to be a spherical Type 1 action. By a similar argument as above, it follows that $\sys(D) \leq \ell =2\pi \sinh R_1$ where $R_1$ is the injectivity radius of any point in the compatible orbit of $D_1$.
\end{remark}

\noindent By Corollary~\ref{cor:bound_inj_pD}, Remark~\ref{rem:inj_bound_rscomp},  and the fact that $\inj(D)=\frac{1}{2} \sys(D)$, we have the following result. 

\begin{corollary} \label{prop:systole1}
Let $D \in \Mod(S_g)$ be such that $D= \lp D_1,D_2,(r,s) \rp$, where the $D_i$ are spherical Type 1 actions.  Then for any $X \in \Fix(\langle D \rangle)$, we have
$$\inj(X) \leq \pi \min\{\sinh(\U(\p_{D_1})), \sinh(\U(\p_{D_2}))\}.$$
\end{corollary}

\begin{remark}
\label{rem:inj_rad_bound}
Let $D$ be a cyclic action on $S_g$ with $n(D) \geq 3$. By (the proof of) Theorem~\ref{thm:arb_cyc_real}, there exists a necklace 
\[\N = ((D_1,\ldots,D_k); ((x_1,y_1), \ldots, (x_m,y_m)); (g',g''))\]
such that $D_{\N} = D$. If $\N$ is just a linear chain (as in Definition \ref{def: compatibility-gen}), then for any hyperbolic structure $X \in \Fix(\langle D \rangle)$ realized by $\N$, an inductive application of Corollary \ref{prop:systole1} shows that $$\sys(X) \leq \U(\N),$$ and thus $$\inj(X) \leq \frac{1}{2}\U(\N),$$ where $$\U(\N):= 2\pi\min\{\sinh(\U(\p_{D_1})), \ldots, \sinh(\U(\p_{D_k})).$$ Moreover, if $\N$ also includes self-compatibilities and the additions (and deletions) of permutation components, Remarks~\ref{rem:inj_bound} and~\ref{rem:inj_bound_rscomp} would guarantee that $\frac{1}{2}\U(\N)$ continues to be an upper bound for $\inj(X)$. 



\end{remark}


\begin{theorem} \label{thm:systole1}
Let $D$ be a cyclic action on $S_g$ with $n(D) \geq 3$, and let $\N$ be a necklace as in Theorem~\ref{thm:main} such that $D_{\N} = D$. Then for any $X \in \Fix(\langle D \rangle)$ realized by $\N$, we have 
$$\sys(X) \leq  \U(\N),$$ and consequently, 
$$\inj(X) \leq \frac{1}{2}\U(\N).$$
\end{theorem}

\noindent The following example illustrates a certain optimality of the bounds obtained in Theorem~\ref{thm:systole1}.

\begin{example}
\label{eg:inj_bd_real1}
For a prime $p \geq 5$, consider the $C_p$ action $D=(p, 0; (p-4,p),(2,p),
(1,p),(1,p))$ on $S_{p-1}$. The action $D$ can be realized as $D_{\N}$, where 
$\N = ((D_1,D_2);;)$ with $\mathfrak{C}(\N) = (((1,1));)$, $$D_1=(p,0;(3,p),
(p-4,p),(1,p)), \text{ and } D_2=(p,0;(p-3,p),(2,p),(1,p)).$$ It follows from 
Proposition \ref{prop:inj_spherical_type1} that $\inj_{P_1}(\mathcal{P}_{D_1})=\inj_{P_2}(\mathcal{P}_{D_2})=R$, where $P_1,P_2$ denote the compatible 
fixed points of $D_1$ and $D_2$ respectively. Since the maximum radius of any 
disk centered at $P_i$ is strictly bounded above by 
$R$, for any structure $X \in \Fix(\langle D \rangle)$,  $\inj(X) < \pi \sinh 
R$ (Corollary \ref{prop:systole1}). However, the upper bound $R$ is optimal in 
the sense that for each $\epsilon>0,$ there exist a point $P\in S_{p-1}$ and a 
structure $X \in \Fix(\langle D \rangle)$ such that the injectivity radius $
\inj_P(X)$ at the point $P$ is $>R-\epsilon.$
\end{example}


Let $\N(D)$ be the set of all necklaces realizing an arbitrary cyclic action $D$, that is
$$\N(D) : = \{\N : D_{\N} = D\}.$$
Clearly, $\N(D)$ is finite as $\Fix(\langle D \rangle)$ is finite-dimensional. We define 
$$\U(D) := \max \{\U(\N): \N \in \N(D)\}.$$
Now, let $H < \Mod(S_g)$ be a finite subgroup such that $H = \langle D_1, \ldots, D_s \rangle$. Then a simple argument would show that $\Fix(H) = \cap_{i=1}^s \Fix(\langle D_i \rangle)$. Define 
$$\U(H) := \min_{1 \leq i \leq s} \U(D_i).$$ As an immediate consequence of Theorem~\ref{thm:systole1}, we obtain the following.

\begin{corollary}
\label{cor:inj_global_bound}
Let $H= \langle D_1, \ldots, D_s \rangle < \Mod(S_g)$ be a finite subgroup such that $n(D_i) \geq 3$ for some $1 \leq i \leq s$. Then for any $X \in \Fix(H)$, we have 
$$\sys(X) \leq  \U(H),$$ and consequently, 
$$\inj(X) \leq \frac{1}{2}\U(H).$$
\end{corollary}  

 Given a closed hyperbolic surface of genus $g$, ($g \geq 2$), the \textit{systole function} denoted by $\sys \colon \Teich(S_g) \to \mathbb{R}_{+}$, is defined by
$\sys \colon (X,\xi) \mapsto \sys(X).$ Finally, as a consequence of Corollary~\ref{cor:inj_global_bound}, we obtain a global upper bound for the systole function restricted to the submanifold $\Fix(H)$ of $\Teich(S_g)$. 

\begin{corollary}\label{cor:systole2}
Let $H < \Mod(S_g)$ be a finite subgroup as in Corollary \ref{cor:inj_global_bound}. Then the restriction $\sys:\Fix(H) \to \mathbb{R}_{+}$ of the systole function is bounded above by $ \U(H)$.
\end{corollary}

\subsection{$\Fix(H)$ as a symplectic submanifold of \textbf{$\Teich(S_g)$}}
As $(\Teich(S_g),\linebreak\mathfrak{g}_{WP})$ is a K\"ahler manifold, where $\mathfrak{g}_{WP}$ denotes Weil-Petersson metric, it admits a canonical symplectic structure given by $\omega =\sum_{i=1}^{3g-3} d\ell_i \wedge d\theta_i$, where $P=(X,\xi)\in \Teich(S_g)$ is an arbitrary point and $(\ell_1,\theta_1,\ldots,\ell_{3g-3},\theta_{3g-3})$ are the Fenchel-Nielsen coordinates of $P$. This follows from the \textit{magic formula} due to Wolpert (\cite{FM,WPS}). Moreover, using the Fenchel-Nielsel coordinate system, it is easy to see that the symplectic manifold $(\Teich(S_g), \omega)$ is symplectomorphic to the Euclidean space $\mathbb{R}^{6g-6}$ with its standard symplectic structure.  

\noindent For any finite subgroup $H < \Mod(S_g)$, $\Fix(H) \approx \Teich(S_g/H) \approx \mathbb{R}^{2k}$ (\cite[Theorem 2]{H1}), which is a K\"ahler submanifold of $(\Teich(S_g),\mathfrak{g}_{WP})$, is also a symplectic submanifold of $(\Teich(S_g), \omega)$. Therefore, it is natural to ask if it is symplectomorphic to $\mathbb{R}^{2k}$ with the standard symplectic structure. The answer to the above question is obtained as an application of Theorem~\ref{thm:main} which provides an explicit embedding of $\Fix(H)$ as a K\"ahler submanifold of $(\Teich(S_g),\mathfrak{g}_{WP})$ when $H =\langle D \rangle$.
\begin{corollary} 
Let $H<\Mod(S_g)$ be a finite subgroup as in Corollaries \ref{cor:inj_global_bound} -\ref{cor:systole2}. Then $\Fix(H)$ is not symplectomorphic to the standard Euclidean space of the corresponding dimension.  
\end{corollary}
\begin{proof}
Using Theorems ~\ref{thm:main} and \ref{thm:systole1}, it follows that 
$$\Fix(\langle D \rangle) =\prod_{j=1}^k M_j,$$
where for some $j$, $M_j=(0,\ell_j) \times \mathbb{R}$, that is, $\Fix(\langle D \rangle)$ can be expressed as product of two-dimensional strips $M_j$ via the global system of Fenchel-Nielsen coordinates, where at least one of the strips is bounded by $\ell_j >0$ and the bound is determined by the irreducible components of $D.$ Moreover, in the same coordinate system, $(\Fix(\langle D \rangle),\omega|_{\Fix(\langle D \rangle)})$ can be realized as a symplectic submanifold with bounded symplectic width (also known as Gromov symplectic capacity) of the Euclidean space $\mathbb{R}^{2k}$ equipped with the standard symplectic structure (see Chapter 12, of \cite{McS} for a formal definition of symplectic width). The desired result is a direct consequence of M. Gromov's symplectic non-squeezing Theorem (see \cite{MG}) which shows that the radii of balls and cylinders are
stored as a symplectic invariant i.e. a ball cannot be embedded into a cylinder via a symplectic map unless the radius of the ball is less than or equal to the radius of the cylinder. As a consequence, in particular, a symplectic manifold of bounded symplectic width cannot be symplectomorphic to the standard Euclidean space of the same dimension as the latter has unbounded symplectic width. 

Finally, for an arbitrary $H= \langle D_1, \ldots, D_s \rangle$, the assertion follows from the fact that $\Fix(H) \approx \cap_{i=1}^s \Fix(\langle D_i \rangle)$, which would also have bounded symplectic width.
\end{proof}

\section{Recovering some well known results}
\label{sec:classical_results}
In this section, we apply our theory to provide alternative proofs to some known results that closely connect with the central theme of this paper. We begin with the following result due to Harvey~\cite{H1,HM1}, which follows as an
immediate consequence of Theorem \ref{thm:main}.
\begin{corollary}\label{cor:dim-branch-loci}
Let $D$ be a cyclic action of order $n$ on $S_g$ such that $\O_D$ has $c$ cone points.  Then $$\dim (\Fix ( \langle D \rangle )) =6g_0(D)+2c-6.$$
\end{corollary}

\begin{proof}
This follows directly from Theorem \ref{thm:main} by observing that $g_0(D_{\N})=g'-g''+m$ and the number of cone points in $\O_{D_{\N}}=k+2f({T_{\N}})-2m+2.$ 
\end{proof}

\noindent Corollary~\ref{cor:dim-branch-loci} leads us to the following result due to Gilman~\cite{JG3} that characterizes the irreducibility of cyclic actions.

\begin{corollary}
\label{cor:gilman}
A cyclic action $D$ on $S_g$ is irreducible if, and only if $g_0(D) = 0$ and $\O_D$ is an orbifold with three cone points. 
\end{corollary}

\begin{proof}
Consider an action $D$ on $S_g$ of the form $D = (n,0;(c_1,n_1),(c_2,n_2),(c_3,n_3)),$ and let $\N$ be any necklace with $k$ beads such that $D_{\N} = D$. It follows from Corollary~\ref{cor:dim-branch-loci} that $\dim(\Fix(\langle D \rangle)) = 0$. Therefore, by Remark~\ref{rem:fixD}, we conclude that $D$ is irreducible.

Conversely, suppose that $D$ is irreducible. Then $g_0(D) = 0$, as otherwise, $D$ would have a nontrivial permutation component. By Remark~\ref{rem:fixD}, it follows $\dim(\Fix(\langle D \rangle)) = 0$, and so Corollary~\ref{cor:dim-branch-loci} would imply that $\O_D$ has exactly 3 cone points, and the assertion follows.
\end{proof}

\subsection{Nielsen realization theorem for cyclic groups} In this subsection, we provide a purely topological proof of the Nielsen realization theorem for cyclic groups motivated by some of the key ideas developed in~\cite{PRS} and this paper. Though the general version of this result (for arbitrary finite subgroups of $\Mod(S_g)$) was proved by Kerckhoff~\cite{SK1}, Fenchel~\cite{WF2,WF1} is credited with giving the first complete proof for cyclic groups, and more generally solvable groups. (The concluding discussion in Nielsen's original paper~\cite{JN1}, settling the cyclic case, was later shown by Zieschang~\cite{HZ1,HZ} to be partly incorrect.) Our proof of this result differs from approaches of Nielsen and Fenchel, as we will use a characterization of irreducible finite-order mapping classes through the orbits of nonseparating curves induced by their actions. A formal statement of the result is as follows: 
\begin{theorem}[Nielsen-Fenchel]
\label{thm:NK}
Let $h \in \Mod(S_g)$ be of order $n$. Then $h$ has a representative $\tilde{h} \in \text{Homeo}^+(S_g)$ such that $\tilde{h}^n = 1$. 
\end{theorem}

\noindent We start by giving a characterization of irreducible Type 1 actions in terms of their curve orbits. We will use $i(a,b)$ (resp. $\tilde{i}(a,b)$) to denote the geometric (resp. algebraic) intersection number of two simple closed curve $a$ and $b$ in $S_g$. 

\begin{proposition}\label{prop:irr-type1_NK}
Let $h \in \Mod(S_g)$ be an irreducible mapping class of order $n$. Then the following statements are equivalent. 
\begin{enumerate}[(i)]
\item There exists a (oriented) nonseparating curve $c$ in $S_g$ with $i(c,h(c)) = \hat{i}(c,h(c))=1$ such that the $\langle h \rangle$-orbit of $c$ is of size $n$.
\item There exists a representative $\tilde{h} \in \text{Homeo}^+(S_g)$ of $h$ that is realized as an order-$n$ rotation of a polygon of type $\p_D$ by $\theta_D$, as in Proposition~\ref{thm:irr_type1}. Equivalently, $\tilde{h}$ generates a spherical Type 1 action on $S_g$. 
\end{enumerate}
\end{proposition}
\begin{proof}
Suppose we assume that (i) holds. Since $i(c,h(c))=1$, we have that $i(h^i(c),h^{i+1}(c)) =1$, for $0 \leq i \leq n-1$, and so it follows that the isotopy class of the curve $\gamma := c*h(c)*\ldots *h^{n-1}(c)$ is represented by a simple closed curve. Moreover, the representation of $\gamma$ together with the fact that $i(h^i(c),h^{i+1}(c)) =1$ would now imply that $h$ preserves the isotopy class of $\gamma$. So, we may assume up to isotopy that $h(\gamma) = \gamma$. As $h$ is irreducible, it follows that $\gamma$ is contractible, and thus $\gamma$ bounds a disk $\p$ in $S_g$. Furthermore, since $\hat{i}(h^i(c),h^{i+1}(c)) =1$,  the orientations on the curves $h^i(c)$ yield a side-pairing on $\partial \p$. Therefore, $\bar{\p}$ is a topological polygon with a side-pairing, which upon identification yields $S_g$. Since $h$ cyclically permutes the curves $\{c,h(c),\ldots,h^{n-1}(c)\}$ in $\partial \p$, it induces a rotation $\theta$ of $\bar{\p}$ of order $n$. Viewing $\bar{\p}$ as a hyperbolic polygon, the irreducibility of $h$ together with Proposition \ref{thm:irr_type1} implies that $\bar{\p}$ is a polygon of type $\p_D$, which $h$ rotates by $\theta_D(= \theta)$. 

Conversely, suppose that $(ii)$ holds. Then the structure of the polygon $\p_D$ guarantees the existence of an orbit $\{c,h(c),\ldots,h^{n-1}(c)\}$ of nonseparating curves (on $\partial \p_D$) as desired.
\end{proof}

 Let $h \in \Mod(S_g)$ be an irreducible finite-order mapping class that is not of Type 1. In~\cite[Lemma 2.23]{PRS}, we had provided a combinatorial recipe for the geometric realization $h$, which involved the attachment of a permutation component followed by a decomposition into Type 1 irreducibles. Taking inspiration from this construction, in the following remark, we sketch a rather technical procedure for realizing such an $h$ topologically.

\begin{remark}
\label{rem:irr_type2}
Let $h \in \Mod(S_g)$ be an irreducible order-$n$ mapping class that is not of Type 1. Assuming, up to isotopy, that $h$ has an orbit of size $n$, we add an $1$-permutation component to $h$ (as in Definition~\ref{rem:triv_self_comp}) along this orbit, to obtain an action $h' = \l h, 1 \r$ on $S_{g + n}$. For $1 \leq i \leq n$, let $a_i$ and $b_i$ denote the longitude and meridional curves on $i^{th}$ copy of $S_{1,0}$ that was pasted to $S_{g}$ during the construction of $h'$. Further, we assume that the for each $i$, the $a_i$ and the $b_i$ are oriented such that $\hat{i}(a_i,b_i) = 1$.  For $1 \leq i \leq n$, let $\beta_i$ be the curve along with the $i^{th}$ copy of $S_{1,0}$ was attached while constructing $h'$. Then the isotopy class of $\beta = \beta_1 \ast \beta_2 \ast \ldots \beta_n$ is represented by a nontrivial curve in $S_g$, for otherwise $\beta$ would bound a disk $\p$ in $S_g$ which a representative $\tilde{h}$ of $h$ rotates by $2\pi/n$, thereby inducing a fixed point in the center of $\p$. This would contradict our assumption that $h$ is not of Type 1. Further, as $h$ irreducible, we have $h(\beta) \neq \beta$, and as $h(\beta_i) = \beta_{i+1}$, it follows that $h(\beta)$ and $\beta$ are homologous. Thus, $h(\beta)$ and $\beta$ cobound a subsurface of $S_g$, and so we have that $i(\beta, h(\beta)) = 0$. 

Now we define $c_i := a_i * b_{i+1}^{-1}$, for $ 1 \leq i \leq n$. Then we see that $i(c_i,c_{i+1})= \hat{i}(c_i,c_{i+1}) = 1$. Further, by definition, it follows that $h'(c_i)= c_{i+1}$, for all $i$. So, $\{c_1,\ldots,c_n\}$ is the curve orbit of $c_1$ under the $\langle h' \rangle$-action such that $i(c_1,f(c_1)) = \hat{i}(c_1,f(c_1)) = 1$. Further,  $h'$ preserves the isotopy class of $c := c_1 \ast c_2 \ast \ldots \ast c_n$ in $S_{g+n}$, which is represented by a simple closed curve. Thus, $h'$ may be isotoped to preserve $c$, and hence a closed neighborhood $N$ of $c$, which is a subsurface $\Sigma$ of $S_{g+n}$. Thus, $h'$ induces an order-$n$ map $h_{\Sigma}$ on $\Sigma$, and so let $h_{\widehat{\Sigma}}$ be the map induced by $h_{\Sigma}$ on the surface $\widehat{\Sigma}$ obtained by capping off the boundary components of $\Sigma$. Since $c$ is contractible in $\widehat{\Sigma}$, by the ideas in Proposition~\ref{prop:irr-type1_NK}, it follows that $h_{\widehat{\Sigma}}$ defines an action on  $\widehat{\Sigma}$ that is realized as a rotation of a polygon with a side-pairing defined by the oriented curves $c_i$. Moreover, $h_{\widehat{\Sigma}}$ must have finitely many $\alpha$-compatibilities (defined along the boundary components of $\Sigma$) with the action $h_{\widehat{\Sigma'}}$ induced by $h'$ on the surface $\widehat{\Sigma}'$ obtained by capping off $\overline{S_{g + n} \setminus \Sigma}$. By removing maximal reduction systems for $h_{\widehat{\Sigma}}$ and $h_{\widehat{\Sigma'}}$ in $\widehat{\Sigma}$ and $\widehat{\Sigma'}$, respectively, and capping (and repeating the process above, if required), we can further decompose $h_{\widehat{\Sigma'}}$ and $h_{\widehat{\Sigma'}}$ into Type 1 irreducibles. 

Finally, by Proposition~\ref{prop:irr-type1_NK}, each of these components has a order $n$ representative that is realized a rotation of a polygon (of type $\p_D$). So, we paste these representatives together across the respective compatibilities to obtain an order-$n$ representative of $h'$. Finally, we delete a $1$-permutation component from $h'$ to obtain an order-$n$ representative of $h$. 
\end{remark}

 By a \textit{multicurve} $\C \subset S_g$, we mean a finite collection of disjoint nonisotopic esential simple closed curves in $S_g$. A multicurve $\C$ in $S_g$ is said to be \textit{nonseparating}, if $S_g \setminus \C$ is connected, and is said to be \textit{separating}, otherwise. Two multicurves $\C$ and $\C'$ in $S_g$ are said to be \textit{mutually disjoint} if there exists no pair of curves $c, c'$ with $c \in \C$ and $c' \in \C'$ such that $i(c,c') > 0$. 

\begin{definition}
Let $h \in \Mod(S_g)$ be of order $n$. Then a multicurve $\C \subset S_g$  is said to form an \textit{essential curve orbit induced by $h$} if:
\begin{enumerate}[(i)]
\item $\C$ is an orbit of some curve $c$ under the action of $\langle h \rangle$, and 
\item $c$ is nonseparating, if $|\C|> 1$, and $c$ is separating, otherwise. 
\end{enumerate}
An essential orbit is said to be of \textit{full size}, if $|\C| =n$.
\end{definition}

\noindent Before we sketch a proof Theorem~\ref{thm:NK}, we will state a lemma (without proof) that gives characterization rotational mapping classes based on the essential curve orbits induced by their actions. 

\begin{lemma}
\label{lem:rot_condn}
Let $h \in \Mod(S_g)$ be of order $n$. Then $h$ is a rotational mapping class if, and only if, $h$ has a maximal reduction system $\C$ such that one of the following holds.
\begin{enumerate}[(i)]
\item There exists $k$ mutually disjoint full sized nonseparating essential curve orbits $\C_i$, for $1 \leq i \leq k$, induced by $h$ such that $k = g/n$, and $\C = \cup_{i=1}^k \C_i$. 
\item There exists $k$ mutually disjoint full sized nonseparating essential curve orbits $\C_i$, for $1 \leq i \leq k$, induced by $h$ such that $k = (g-1)/n$, and there exists a nonseparating curve $c$ in $S_g$ disjoint from the curves in each $\C_i$ such that $\C = \cup_{i=1}^k \C_i \cup \{c\}$. 
\end{enumerate}
\end{lemma}

\noindent We conclude this paper by sketching a proof of the Nielsen realization theorem for the cyclic case.

\begin{proof}[Proof of Theorem~\ref{thm:NK}] 
Let $h\in \Mod(S_g)$ be of order $n$. If $h$ is irreducible, then by Proposition~\ref{prop:irr-type1_NK} and Remark~\ref{rem:irr_type2}, we can obtain a representative $\tilde{h}$ of $h$, of order $n$, as required. Moreover, if $h$ has a maximal reduction system $\C$ that satisfies the condition (i) of Lemma~\ref{lem:rot_condn}. Then by capping off $\overline{S_g \setminus \C}$, $h$ reduces to a map on the sphere $S_0$, which is isotopic to a rotation $R$. Note that the action of $R$ on $S_0$ has two marked orbits of size $n$ on $S_0$ corresponding to each essential full-sized orbit in $\C$ (that was removed). We may inductively paste $k$ $1$-permutations to $R$ along these marked orbits in $S_0$ (corresponding to the full-sized essential orbits in $\C$ we had removed) to obtain an $\tilde{h} \in \text{Homeo}^+(S_g)$, as desired. If $h$ has a maximal reduction system $\C$ that satisfies condition (ii) of Lemma~\ref{lem:rot_condn}, then by a similar argument as above, we may obtain $\tilde{h}$. 

Suppose that $h$ is a reducible non-rotational mapping class. Choose a collection $\{\C_1, \ldots,\C_k\}$ of mutually disjoint essential orbits induced by $h$ such that $\C = \cup_{i=1}^k \C_i$ forms a maximal reduction system for $h$. Let $h'$ be the map induced by $h$ on the surface $S'$ obtained by capping off the components of $\overline{S_g \setminus \C}$. As before, we note that the removal each $\C_i \subset \C$ induces two marked orbits of size $|\C_i|$ in $S'$. By our assumption, it now follows that each component of $S'$ is of genus greater than or equal to $1$, and furthermore, $h'$ induces an irreducible mapping class on each component of $S'$. By Proposition~\ref{prop:irr-type1_NK} and Remark~\ref{rem:irr_type2}, we can obtain an order-$n$ representative of each component of $h'$. Now we paste together these representatives  across $r$-compatibilities along the marked orbits (corresponding to an essential orbits in the collection $\{\C_1, \ldots,\C_k\}$ we had removed), to obtain a representative $\tilde{h}$ of the mapping class $h$ such that $\tilde{h}^n =1$.
\end{proof}

\section*{Acknowledgments} The first and third authors were supported by a joint SERB-EMR grant instituted by the Government of India.
 
\bibliographystyle{plain}
\bibliography{cyclic_realizations}

\end{document}